\documentclass[12pt]{amsart}
\usepackage{latexsym}
\usepackage{amssymb}
\usepackage{amscd}
\usepackage{mathrsfs}
\usepackage[all]{xy}

\renewcommand\a{\alpha}

\newcommand\g{\gamma}
\renewcommand\d{\delta}
\newcommand\la{\lambda}

\newcommand\s{\sigma}

\newcommand\f{\phi}
\newcommand\vf{\varphi}

\newcommand\Om{\Omega}

\newcommand\vG{\varGamma}
\newcommand\ve{\varepsilon}

\newcommand\BA{\mathbf A}

\newcommand\BQ{\mathbf Q}
\newcommand\BF{\mathbf F}

\newcommand\BZ{\mathbf Z}

\newcommand\BN{\mathbf N}

\newcommand\bB{\mathbf B}

\newcommand\BU{\mathbf U}
\newcommand\BV{\mathbf V}

\newcommand\Ba{\mathbf a}
\newcommand\Bq{\mathbf q}

\newcommand\Bs{\mathbf s}

\newcommand\Bk{\mathbf k}

\newcommand\Bell{\boldsymbol\ell}

\newcommand\CB{\mathcal{B}}

\newcommand\SE{\mathscr{E}}

\newcommand\SL{\mathscr{L}}

\newcommand\Fg{\mathfrak g}

\newcommand\iv{^{-1}}
\newcommand\wh{\widehat}
\newcommand\wt{\widetilde}

\newcommand\ol{\overline}
\newcommand\ul{\underline}

\newcommand\dia{\diamondsuit}

\newcommand\lra{\leftrightarrow}

\newcommand\Ker{\operatorname{Ker}}
\newcommand\Hom{\operatorname{Hom}}

\newcommand\lp{\operatorname{\!\langle\!}}
\newcommand\rp{\operatorname{\!\rangle\!}}
\renewcommand\Im{\operatorname{Im}}

\newcommand\weit{\operatorname{wt}}

\newcommand{\isom}{\,\raise2pt\hbox{$\underrightarrow{\sim}$}\,}
\numberwithin{equation}{section}

\newtheorem{thm}{Theorem}[section]
\newtheorem{lem}[thm]{Lemma}
\newtheorem{cor}[thm]{Corollary}
\newtheorem{prop}[thm]{Proposition}

\def \para#1{\par\medskip\textbf{#1}
              \addtocounter{thm}{1}}

\def \remark#1{\par\medskip\noindent
                \textbf{Remark #1}
                \addtocounter{thm}{1}}

\begin{document}
\setlength{\baselineskip}{4.9mm}
\setlength{\abovedisplayskip}{4.5mm}
\setlength{\belowdisplayskip}{4.5mm}
\renewcommand{\theenumi}{\roman{enumi}}
\renewcommand{\labelenumi}{(\theenumi)}
\renewcommand{\thefootnote}{\fnsymbol{footnote}}
\renewcommand{\thefootnote}{\fnsymbol{footnote}}
\allowdisplaybreaks[2]
\parindent=20pt
\medskip
\begin{center}
{\bf Diagram automorphisms and canonical bases \\ 
for quantized enveloping algebras}  
\par
\vspace{1cm}
Ying Ma, Toshiaki Shoji and Zhiping Zhou
\\
\title{}
\end{center}

\begin{abstract}
Let $X$ be a Cartan datum of symmetric type, with 
an admissible automorphism $\s$ on $X$, and $\ul X$ the 
Cartan datum induced from $(X, \s)$.  Let $\BU_q^-$ 
(resp. $\ul \BU_q^-$) be the negative part of the 
quantized enveloping algebra 
associated to $X$ (resp. $\ul X$).  Lusztig constructed 
the canonical basis $\bB$ of $\BU_q^-$ and the canonical 
signed basis $\wt{\ul\CB}$ of $\ul\BU_q^-$ by making use 
of the geometric theory of quivers.  By normalizing the sign
of $\wt{\ul\CB}$, he obtained the canonical basis $\ul\bB$ of $\ul\BU_q^-$, 
and a natural bijection $\bB^{\s} \isom \ul\bB$.
In this paper, assuming the existence of $\bB$ for $\BU_q^-$, 
we construct the canonical basis $\ul\bB$ of $\ul\BU_q^-$, 
and a bijection $\bB^{\s} \isom  \ul\bB$,  
by an elementary method, subject to the condition that 
the order of $\s$ is odd. In the case where the order is even, 
we obtain the corresponding result for the canonical signed basis. 
\end{abstract}

\maketitle
\pagestyle{myheadings}

\begin{center}
{\sc Introduction}
\end{center}

\par
Let $X$ be the Cartan datum corresponding to 
a Kac-Moody algebra $\Fg$, and $\BU_q^-$ the negative part of 
the quantized enveloping algebra $\BU_q = \BU_q(\Fg)$. 
Let $\s : X \to X$ be the diagram automorphism of $X$.
Then $\s$ induces an algebra automorphism $\s: \BU_q^- \to \BU_q^-$.
In the case where $\s$ is admissible (see 2.1 for the definition),
the pair $(X,\s)$ induces a Cartan datum $\ul X$, which 
corresponds to the orbit algebra $\Fg^{\s}$ of $\Fg$.
Let $\ul\BU_q^-$ be the negative part of the quantized enveloping algebra
$\ul\BU_q$ associated to $\ul X$. 
\par
In the case where $X$ is a symmetric Cartan datum, the canonical basis 
$\bB$ of $\BU_q^-$ was constructed by Lusztig [L1] by making use of 
the geometric theory of quivers. 
In [L2], this result was generalized to an arbitrary Cartan datum  $\ul X$.    
The outline of the proof is as follows; for a given Cartan datum $\ul X$, 
there exists a Cartan datum $X$ of symmetric type, and 
an admissible  diagram automorphism $\s : X \to X$ such that the Cartan datum induced 
from $(X, \s)$ is isomorphic to $\ul X$.  Let $\bB$ be the canonical basis of $\BU_q^-$.
Then $\s$ gives a permutation of $\bB$, and we denote by $\bB^{\s}$ the set
of $\s$-fixed elements in $\bB$. 
In the first step, 
by making use of the geometric theory of quivers with automorphisms, 
he constructed the canonical signed basis $\wt{\ul \CB}$ of $\ul\BU_q^-$ 
(see 1.18), and 
proved that there exists a natural bijection $\wt\bB^{\s} \isom \wt{\ul\CB}$, 
where $\wt\bB = \bB \sqcup -\bB$, 
and $\wt{\ul\CB} = \ul\CB \sqcup -\ul\CB$ for some 
basis $\ul\CB$ of $\ul\BU_q^-$ ([L2, Thm. 14.2.3]).
In the second step, he proved the existence of the canonical 
basis $\ul\bB$ of $\ul\BU_q^-$ such that 
$\wt{\ul\CB} = \ul\bB \sqcup - \ul\bB$ and constructed a natural bijection 
$\bB^{\s} \isom \ul\bB$ ([L2, Thm. 14.4.3, 19.2.3]), 
by normalizing the sign of $\wt{\ul\bB}$ in terms of 
the theory of crystals due to Kashiwara [K].
\par
In this paper, we take up a similar problem, but from a different point 
of view. This is a generalization of [SZ1,2], where the case $X$ is finite or 
affine type was discussed. We consider $X$ of symmetric type, with an automorphism $\s$ on $X$,    
and $\ul X$ induced from $(X,\s)$ as before. Assuming the existence of 
the canonical basis $\bB$ for $\BU_q^-$, we shall construct the 
canonical signed basis $\wt{\ul\CB}$ of 
$\ul\BU_q^-$ and a natural bijection $\wt\bB^{\s} \isom \wt{\ul\CB}$ in an elementary way, 
without appealing the geometric theory, nor the theory of crystal basis.
In the case where the order of $\s$ is odd, we obtain the canonical basis $\ul\bB$ 
of $\ul\BU_q^-$ 
and a natural bijection $\bB^{\s} \isom \ul\bB$.   
\par
The main ingredient for our approach is the isomorphism 
$\Phi: {}_{\BA'}\ul\BU_q^- \isom \BV_q$ as discussed in [SZ1, SZ2], where 
$X$ is assumed to be finite type or affine type.  
We generalize those results to the following situation. 
Let $X$ be a Cartan datum of arbitrary type, with an admissible 
automorphism $\s$ on $X$. Then the pair $(X, \s)$ induces 
a Cartan datum $\ul X$.  We consider the quantized enveloping algebra 
$\BU_q^-$ (resp. $\ul\BU_q^-$) associated to $X$ (resp. $\ul X$).
Here we assume that the order of $\s$ 
is a power of a prime number $p$.  Let $\BA = \BZ[q, q\iv]$, and 
$\BA' = \BF[q,q\iv]$, where $\BF = \BZ/p\BZ$ is the finite 
field of $p$-elements.
Let ${}_{\BA}\BU_q^-$ be Lusztig's integral form of $\BU_q^-$, 
and set ${}_{\BA'}\BU_q^- = \BA' \otimes_{\BA}{}_{\BA}\BU_q^-$. 
$\s$ acts naturally on ${}_{\BA'}\BU_q^-$, and we denote by 
${}_{\BA'}\BU_q^{-,\s}$ the $\s$-fixed point subalgebra of ${}_{\BA'}\BU_q^-$.   
The $\BA'$-algebra $\BV_q$ is defined as $\BV_q = {}_{\BA'}\BU_q^{-,\s}/J$, 
where $J$ is the two-sided ideal generated by the orbit sum $O(x)$ such that 
$\s(x) \ne x$ (see 3.2).  The $\BA'$-algebra ${}_{\BA'}\ul\BU_q^-$ 
is defined similarly to the case of $\BU_q^-$. 
In turn, we give an axiomatic definition of the canonical basis 
for $\BU_q^-$ (see 1.12) by focussing some properties of the canonical basis
in the symmetric case.  It is shown that the canonical basis is 
unique if it exists.  We show, in Theorem 3.4, that 
if the canonical (signed) basis $\bB$ exists for $\BU_q^-$, then there exists an isomorphism  
$\Phi : {}_{\BA'}\ul\BU_q^- \isom \BV_q$ of $\BA'$-algebras.  
Moreover, in Theorem 4.18, under the assumption that $p$ is odd, 
we construct the canonical basis $\ul\bB$ of $\ul\BU_q^-$ from $\bB$, 
by making use of this isomorphism, 
and show that there exists a natural bijection $\bB^{\s} \isom \ul\bB$. 
In the case where $p = 2$, we only obtain the canonical signed basis
(see Remark 4.19). 
\par
Returning to the original problem, we consider $X$ of symmetric type, and 
$\s : X \to X$ such that $\ul X$ is induced from $(X,\s)$.
In this case, the order of $\s$ is not necessarily a prime power, 
but one can find a sequence $X = X_0, X_1, \dots, X_k = \ul X$ of Cartan data, 
and an automorphism $\s_i : X_i \to X_i$ such that $X_{i+1}$ is isomorphic 
to the Cartan datum induced from $(X_i, \s_i)$ and that the order of $\s_i$ 
is a prime power, with $\s = \s_{k-1}\cdots\s_1\s_0$. 
In the case where the order of $\s$ is odd, since 
the canonical basis $\bB$ exists for $X_0 = X$, by the repeated use of Theorem 4.18, 
one can find the canonical basis $\ul\bB$ of $X_k = \ul X$, and a bijection 
$\bB^{\s} \isom \ul\bB$ (Theorem 4.27). 
While in the case where the order of $\s$ is even, 
we obtain the corresponding result concerning with the canonical signed basis. 
\par
In the course of the proof of Theorem 3.4, one needs to show that $\Phi$ is 
a homomorphism. In [SZ1], where $X$ is finite or affine type, 
this was proved by case by case verification, 
by making use of PBW-bases. In our case, we can not use PBW-bases. Instead, 
in this paper we prove this by a purely combinatorial argument, 
in a uniform way.  The discussion here is in some sense simpler, 
and more transparent  than the direct computation in [SZ1]. 
\par
The authors would like to thank the referee for the careful reading, 
and for informing them Tanisaki's result [T]
on integrable highest weight modules in positive characteristic.  
\section{Preliminaries}

\para{1.1.}
Let $X = (I, (\ ,\ ))$ be a Cartan datum, 
where $I$ is a finite set and 
$(\ ,\ )$ is a symmetric bilinear form on the vector space 
$\bigoplus_{i \in I}\BQ\a_i$ with basis $\a_i$, such that $(\a_i,\a_j) \in \BZ$
satisfies the property
\begin{enumerate}
\item \
$(\a_i,\a_i) \in 2\BZ_{>0}$ for any $i \in I$, 
\item \
$\frac{2(\a_i,\a_j)}{(\a_i,\a_i)} \in \BZ_{\le 0}$ for any $i \ne j$ in $I$. 
\end{enumerate}
For $i,j \in I$, set $a_{ij} = 2(\a_i,\a_j)/(\a_i,\a_i) \in \BZ$.
The matrix $A = (a_{ij})$ is called the Cartan matrix associated to $X$. 
The Cartan datum is said to be symmetric if $(\a_i,\a_i) = 2$ for any $i \in I$, 
and simply-laced if it is symmetric and $(\a_i, \a_j) \in \{ 0, -1\}$
for any $i \ne j$.  If $X$ is symmetric, then $A$ is a symmetric matrix. 
Let $Q = \bigoplus_{i \in I}\BZ \a_i$ be the root lattice, and set 
$Q_+ = \sum_{i \in I}\BN \a_i, Q_- = - Q_+$. 
For a weight $\nu = \sum_{i \in I} n_i\a_i \in Q_{\pm}$,
set $|\nu| = \sum_i |n_i|$.  
Let $\Fg$ be the Kac-Moody algebra associated to $X$.
\par
We fix  a weight lattice $P$, a free abelian group of finite rank, 
 such that $Q \subset P$, and set 
$P^* = \Hom_{\BZ}(P, \BZ)$.  Let $\lp\ ,\ \rp : P^* \times P \to \BZ$ 
be the canonical pairing.  We fix $h_i \in P^*$ for each $i \in I$ satisfying 
the property that $\lp h_i, \la\rp = \frac{2(\la, \a_i)}{(\a_i, \a_i)}$ 
for any $\la \in P$.   

\para{1.2.}
Let $q$ be an indeterminate, and for an integer $n$, a positive integer $m$, 
set
\begin{equation*}
[n] = \frac{q^n - q^{-n}}{q -q\iv}, \quad [m]^! = [1][2]\cdots [m], 
    \quad [0]^! = 1.
\end{equation*}
Also, for $n \in \BZ, m  \in \BN$, put
\begin{equation*}
\begin{bmatrix}
         n  \\
         m
\end{bmatrix} = \frac{[n][n-1]\cdots [n-m +1]}{[m]^!}, \quad (m \ge 1), 
\qquad \begin{bmatrix}
         n  \\
         0
\end{bmatrix} = 1. 
\end{equation*}
In particular, if $0 \le m \le n$, then we have 
$\begin{bmatrix}
         n  \\
         m
 \end{bmatrix} = \frac{[n]^!}{[m]^![n-m]^!} = 
\begin{bmatrix}
         n  \\
         n - m
\end{bmatrix}$.
\par
For $d \in \BN$, we denote by $[n]_d$ the element obtained from $[n]$ by 
replacing $q$ by $q^d$.  
For each $i \in I$, set $d_i = (\a_i,\a_i)/2 \in \BN$, and 
$q_i = q^{d_i}$. 
\par
Let $\BU_q = \BU_q(\Fg)$ be the quantized enveloping algebra associated to 
$X$ and $P$, namely, an associative algebra over $\BQ(q)$ with generators
$e_i, f_i \ (i \in I)$ and $q^h \ (h \in P^*)$, and relations 
\begin{align*}
\tag{1.2.1}
&q^0 = 1, \qquad q^{h + h'} = q^hq^{h'}  \quad\text{ for $h,h' \in P^*$, }    \\
\tag{1.2.2}
&q^he_iq^{-h} = q^{\lp h,\a_i\rp} e_i, \qquad q^hf_iq^{-h} = q^{-\lp h,\a_i\rp}f_i
                  \quad\text{ for $i \in I, h \in P^*$,}  \\
\tag{1.2.3}
&e_if_j - f_je_i = \d_{ij}\frac{t_i - t_i\iv}{q_i - q_i\iv}  \quad\text{ for $i,j \in I$,} \\
\tag{1.2.4}
&\sum_{k = 0}^{1 - a_{ij}}(-1)^ke_i^{(k)}e_je_i^{(k)} = 0, 
\quad 
\sum_{k = 0}^{1 - a_{ij}}(-1)^kf_i^{(k)}f_jf_i^{(k)} = 0
\quad\text{ for $i \ne j \in I$,}
\end{align*} 
where $t_i = q^{d_ih_i}$, and $e_i^{(n)} = e_i^n/[n]^!_{d_i}, 
f_i^{(n)} = f_i^n/[n]^!_{d_i}$.
Let $\BU_q^-$ (resp. $\BU_q^+$) be the subalgebra of $\BU_q$ 
generated by $f_i\ (i \in I)$ (resp. by $e_i \ (i \in I)$).  
Then $\BU_q^-$ (resp. $\BU_q^+$) is an associative algebra over $\BQ(q)$ with generators 
$f_i$ (resp. $e_i$)  satisfying the fundamental relations in (1.2.4).
\par   
Set $\BA = \BZ[q,q\iv]$, and let ${}_{\BA}\BU_q^-$ be Lusztig's integral form 
of $\BU_q^-$, namely, the $\BA$-subalgebra of $\BU_q^-$ generated by 
$f_i^{(n)} = f_i^n/[n]^!_{d_i}$ for $i \in I, n \in \BN$.   
\par
We define a $\BQ$-algebra automorphism, called the bar-involution, 
${}^-: \BU_q^- \to \BU_q^-$ by $\ol q = q\iv$, $\ol f_i = f_i$ for $i \in I$.
We define an anti-algebra automorphism ${}^* : \BU_q^- \to \BU_q^-$
by $f_i^* = f_i$ for any $i \in I$.  
\para{1.3.}
$\BU_q^-$ has a weight space decomposition 
$\BU_q^- = \bigoplus_{\nu \in Q_-}(\BU_q^-)_{\nu}$, where 
$(\BU_q^-)_{\nu}$ is a subspace of $\BU_q^-$ spanned by elements $f_{i_1}\cdots f_{i_N}$
such that $\a_{i_1} + \cdots + \a_{i_N} = -\nu$. 
$x \in \BU_q^-$ is called homogeneous with $\weit x = \nu$ if $x \in (\BU_q^-)_{\nu}$. 
We define a multiplication on $\BU_q^- \otimes \BU_q^-$ by 

\begin{equation*}
\tag{1.3.1}
(x_1\otimes x_2)\cdot (x_1'\otimes x_2') = q^{-(\weit x_2, \weit x_1')}
                                             x_1x_1'\otimes x_2x_2' 
\end{equation*}
where $x_1, x_1', x_2, x_2'$ are homogeneous in $\BU_q^-$. 
Then $\BU_q^-\otimes \BU_q^-$ becomes an associative algebra with respect to
this twisted product. We define a homomorphism 
$r : \BU_q^- \to \BU_q^-\otimes \BU_q^-$ by $r(f_i) = f_i\otimes 1 + 1 \otimes f_i$ 
for each $i \in I$. It is known (see [L2, 1.2]. Note that 
$v_i$ in [L2] coincides with $q_i\iv$ in our paper) 
that there exists a unique bilinear form 
$(\ ,\ )$ on $\BU_q^-$ satisfying the following properties; 
$(1,1) = 1$ and

\begin{equation*} 
\tag{1.3.2}
\begin{aligned}
(f_i, f_j) = \d_{ij}(1 - q_i^2)\iv, \\
(x, y'y'') = (r(x), y'\otimes y''), \\
(x'x'', y) = (x'\otimes x'', r(y)),
\end{aligned}
\end{equation*}
where the bilinear form on $\BU_q^-\otimes \BU_q^-$ is
defined by $(x_1\otimes x_2, x_1'\otimes x_2') = (x_1, x_1')(x_2, x_2')$. 
Thus defined bilinear form is symmetric and non-degenerate.
Using the property 
$r((\BU_q^-)_{\nu}) \subset \bigoplus_{\nu' + \nu'' 
         = \nu}(\BU_q^-)_{\nu'}\otimes (\BU_q^-)_{\nu''}$, 
we have  
\begin{equation*}
\tag{1.3.3}
((\BU_q^-)_{\nu}, (\BU_q^-)_{\nu'}) = 0  \quad\text{ for } \nu \ne \nu'.
\end{equation*}
\par
For any $i \in I$, we define $\BQ(q)$-linear maps ${}_ir, r_i : \BU_q^- \to \BU_q^-$
by 
\begin{equation*}
\tag{1.3.4}
r(x) = f_i \otimes {}_ir(x) + \sum y\otimes z, \qquad
r(x) = r_i(x)\otimes f_i + \sum z\otimes y,
\end{equation*}
where $y$ are homogenous such that $\weit y \ne -\a_i$. 
From the definition, we have
\begin{equation*}
\tag{1.3.5}
(f_iy, x) = (f_i,f_i)(y, {}_ir(x)), \qquad (yf_i, x) = (f_i, f_i)(y, r_i(x)).
\end{equation*}

\par
The following properties for ${}_ir, r_i$ are also immediate from the definition. Assume that 
$x,x'$ are homogeneous.  Then
\begin{equation*}
\tag{1.3.6}
\begin{aligned}
&{}_ir(1) = 0, \quad {}_ir(f_j) = \d_{ij},  
\qquad r_i(1) = 0, \quad r_i(f_j) = \d_{ij},  \\
&{}_ir(xx') = q^{(\weit x, \a_i)}x\,{}_ir(x') + {}_ir(x)x', \quad
 r_i(xx') = q^{(\weit x', \a_i)}r_i(x)x' + xr_i(x').
\end{aligned}
\end{equation*}

Moreover, we have 
\begin{equation*}
\tag{1.3.7}
r_i = * \circ {}_ir \circ *.
\end{equation*}

\begin{lem}  
Assume that $(\a_i,\a_j) = 0$.  Then 
\begin{enumerate}
\item 
${}_ir$ commutes with the left 
action of $f_j$ on $\BU_q^-$.
\item ${}_ir$ and ${}_jr$ commute each other.
\end{enumerate}
\end{lem}

\begin{proof}
(i) is immediate from (1.3.6).
We show (ii). Assume that $x$ is homogeneous.  
We prove (*) ${}_ir{}_jr(x) = {}_jr{}_ir(x)$ by induction on $|\weit x|$. 
If $x = 1$, this is trivial.  So assume that $x \ne 0$, and (*) holds for 
$x'$ such that $|\weit x'| < |\weit x|$.  Write $x = yz$ with $y, z$ homogeneous, 
not equal to 1.
Then by (1.3.6), we have
${}_ir(yz) = q^{(\weit y, \a_i)}y{}_ir(z) + {}_ir(y)z$, and so   
\begin{align*}
{}_jr{}_ir(yz) &= q^{(\weit y, \a_i)}
           \biggl(q^{(\weit y, \a_j)}y{}_jr{}_ir(z) + {}_jr(y){}_ir(z)\biggr)  \\
     &\qquad + \biggl(q^{(\weit y + \a_i, \a_j)}{}_ir(y){}_jr(z)  + {}_jr{}_ir(y)z\biggr) \\
     &= q^{(\weit y, \a_i) + (\weit y, \a_j)}y{}_jr{}_ir(z) \\
     &\qquad + \biggl(q^{(\weit y,\a_i)}{}_jr(y){}_ir(z) 
                   + q^{(\weit y, \a_j)}{}_ir(y){}_jr(z)\biggr) + {}_jr{}_ir(y)z.    
\end{align*}
In the last formula, 
${}_jr{}_ir(y) = {}_ir{}_jr(y),  {}_jr{}_ir(z) = {}_ir{}_jr(z)$ by induction 
hypothesis, and the second term is symmetric with respect to $i$ and $j$.  Thus 
we have ${}_jr{}_ir(yz) = {}_ir{}_jr(yz)$.  Hence (ii) holds. 
\end{proof}

The following result is known (cf. [L2, Lemma 1.2.15]).

\begin{lem}  
\ $\bigcap_{i \in I}\Ker{}_ir = \BQ(q)1, \qquad 
   \bigcap_{i \in I}\Ker r_i = \BQ(q)1$. 
\end{lem}

\para{1.6.}
The following formula is easily verified by induction on $n$ 
(see [L2, Lemma 1.4.2]). 
\begin{equation*}
\tag{1.6.1}
r(f_i^{(n)}) = \sum_{k + k' = n}q_i^{-kk'}
                          f_i^{(k)}\otimes f_i^{(k')}.
\end{equation*}
It follows from (1.6.1) that $r : \BU_q^- \to \BU_q^-\otimes \BU_q^-$
induces a homomorphism 
$r : {}_{\BA}\BU_q^- \to {}_{\BA}\BU_q^- \otimes_{\BA}{}_{\BA}\BU_q^-$. 
This implies that the maps ${}_ir, r_i : \BU_q^- \to \BU_q^-$ induce 
$\BA$-linear maps ${}_ir, r_i :{}_{\BA}\BU_q^- \to {}_{\BA}\BU_q^-$. 

\para{1.7.}
For $i \in I$ and any $t \ge 0$, we consider the operator 
\begin{equation*}
\tag{1.7.1}
\Pi_{i,t} = \sum_{s \ge 0}(-1)^sq_i^{-s(s-1)/2}f_i^{(s)}({}_ir)^{s+t} : \BU_q^- \to \BU_q^-. 
\end{equation*} 
The following result is known by [L2, Lemma 16.1.2] (originally due to 
Kashiwara [K, 3.2]). For the orthogonality relations, see (1.9.2) below.
For $i \in I$, set $\Ker {}_ir = \BU_q^-[i]$. 
Then $\Ker r_i = *(\BU_q^-[i])$ which we denote by ${}^*\BU_q^-[i]$. 
Note that the statements for ${}^*\BU_q^-[i]$ follows from that for 
$\BU_q^-[i]$ by applying the $*$-operation. 
\begin{lem}  
For $i \in I$, the followings holds. 
\begin{enumerate}
\item \ $\BU_q^- = \bigoplus_{n \ge 0}f_i^{(n)}\BU_q^-[i] 
          = \bigoplus_{n \ge 0}{}^*\BU_q^-[i]f_i^{(n)}$ as vector spaces.
The direct summands $f_i^{(n)}\BU_q^-[i]$ are mutually orthogonal with respect
$(\ ,\ )$. A similar result holds for ${}^*\BU_q^-[i]f_i^{(n)}$.
\item \ 
The map $x \mapsto f_i^{(n)}x$ gives an isomorphism 
$\BU_q^-[i] \isom f_i^{(n)}\BU_q^-[i]$, and similarly, the map 
$x \mapsto xf_i^{(n)}$ gives an isomorphism 
${}^*\BU_q^-[i] \isom {}^*\BU_q^-[i]f_i^{(n)}$. 
\item \ 
By (i) and (ii), $x \in \BU_q^-$ is written uniquely as 
$x = \sum_{n \ge 0}f_i^{(n)}x_n$ with $x_n \in \BU_q^-[i]$. Then 
\begin{equation*}
\tag{1.8.1}
x_n = q_i^{n(n-1)/2}\Pi_{i,n}(x).
\end{equation*} 
In particular, the projection $\BU_q^- \to f_i^{(n)}\BU_q^-[i]$ preserves 
the weights, namely, if $x \in (\BU_q^-)_{\nu}$, then 
$f_i^{(n)}x_n \in (\BU_q^-)_{\nu}$ for any $n \ge 0$. 
\end{enumerate}
\end{lem}

\para{1.9.}
For any subspace $Z$ of $\BU_q^-$, set ${}_{\BA}Z = Z \cap {}_{\BA}\BU_q^-$.
Take $x \in \BU_q^-$. then $x$ is written uniquely as 
$x = \sum_{n \ge 0}f_i^{(n)}x_n$ with $x_n \in \BU_q^-[i]$ by Lemma 1.8.
Since ${}_ir$ gives a map ${}_{\BA}\BU_q^- \to {}_{\BA}\BU_q^-$ by 1.6, 
$\Pi_{i,t}$ induces a map ${}_{\BA}\BU_q^- \to {}_{\BA}\BU_q^-$. 
Assume that $x \in {}_{\BA}\BU_q^-$.  Then by (1.8.1), $x_n \in {}_{\BA}\BU_q^-$, 
and so $f_i^{(n)}x_n \in {}_{\BA}\BU_q^-$.  
It follows that the decomposition in Lemma 1.8 (i) induces a direct sum decomposition
\begin{equation*}
\tag{1.9.1}
{}_{\BA}\BU_q^- = \bigoplus_{n \ge 0}{}_{\BA}(f_i^{(n)}\BU_q^-[i])
                = \bigoplus_{n \ge 0}f_i^{(n)}{}_{\BA}(\BU_q^-[i]),  
\end{equation*}
where ${}_{\BA}(f_i^{(n)}\BU_q^-[i]) = f_i^{(n)}{}_{\BA}\BU_q^-[i]$. 

\par
The following orthogonality relations hold 
for the decomposition in (1.9.1). For the proof, see [L2, Lemma 16.2.6], 
where the case $m = n$ is discussed. The case $m \ne n$ is also treated 
by a similar argument. 

\par\medskip\noindent
(1.9.2) \ Assume that $x, y \in {}_{\BA}\BU_q^-[i]$.  Then  
\begin{equation*}
(f_i^{(n)}x, f_i^{(m)}y) = \begin{cases}
                            c(x, y) \text{ with } c \in 1 + (q\BZ[[q]] \cap \BQ(q)), 
                              &\quad\text{ if $m = n$, } \\
                             0    &\quad\text{ if $m \ne n$.}
                           \end{cases}  
\end{equation*}
 
\para{1.10.}
Let $V$ be a $\BQ(q)$-subspace of $\BU_q^-$.
A basis $\CB$ of $V$ is said to be almost orthonormal if
\begin{equation*}
(b, b') \in \begin{cases}
               1 + q\BZ[[q]] \cap \BQ(q) &\quad\text{ if } b = b', \\
               q\BZ[[q]] \cap \BQ(q)     &\quad\text{ if } b \ne b'.
             \end{cases}
\end{equation*}

Recall that $\BA = \BZ[q,q\iv]$.  Let $\BA_0 = \BQ[[q]] \cap \BQ(q)$. 
Set
\begin{equation*}
\tag{1.10.1}
\SL_{\BZ}(\infty) = \{ x \in {}_{\BA}\BU_q^- \mid (x,x) \in \BA_0\}.
\end{equation*}
Then $\SL_{\BZ}(\infty)$ is a $\BZ[q]$-submodule of ${}_{\BA}\BU_q^-$. 
It is known that if $\CB$ is an $\BA$-basis of ${}_{\BA}\BU_q^-$, which is almost orthonormal, 
then $\CB$ gives a $\BZ[q]$-basis of $\SL_{\BZ}(\infty)$ by [L2, Lemma 16.2.5].  

\par
For a fixed $i \in I$, we consider the decomposition 
 ${}_{\BA}\BU_q^- = \bigoplus_{n \ge 0}f_i^{(n)}{}_{\BA}\BU_q^-[i]$ as in (1.9.1). 
Write $x = \sum_{n \ge 0}y_n = \sum_{n \ge 0}f_i^{(n)}x_n$, where 
$y_n = f_i^{(n)}x_n$ with $x_n \in {}_{\BA}\BU_q^-[i]$. 
The following result is known.

\begin{lem}[{[L2, Lemma 16.2.7]}]  
Let $x = \sum_{n \ge 0}y_n$ be as above. 
\begin{enumerate}
\item 
Assume that $x \in \SL_{\BZ}(\infty)$. Then each $x_n$ and $y_n$ is in $\SL_{\BZ}(\infty)$. 
If, in addition, $(x,x) \in 1 + q\BA_0$, then there exists $n_0 \ge 0$ such that
$(y_{n_0}, y_{n_0}), (x_{n_0}, x_{n_0}) \in 1 + q\BA_0$ and 
$(y_n, y_n), (x_n, x_n) \in q\BA_0$ for all $n \ne n_0$.  
\item 
Assume that $\CB$ is an $\BA$-basis of ${}_{\BA}\BU_q^-$, which is almost orthonormal. 
Then in the setup of (i), 
$y_{n_0} \equiv \pm b \mod q\SL_{\BZ}(\infty)$ for some $b \in \CB$, and 
$x_n \equiv 0 \mod q\SL_{\BZ}(\infty), y_n \equiv 0 \mod q\SL_{\BZ}(\infty)$ 
for all $n \ne n_0$. 
\end{enumerate}
\end{lem}

\para{1.12.}
For a fixed $i \in I$, we consider the direct sum  decomposition 
of ${}_{\BA}\BU_q^-$ as in (1.9.1). For each $x \in \BU_q^-$, let 
$\ve_i(x)$ be the largest integer $n$ such that $x \in f_i^{(n)}\BU_q$, and
$x_{[i;a]}$ the projection of $x$ on $f_i^{(a)}\BU_q^-[i]$ for $a \in \BN$. 
By Lemma 1.11, if $x \in \SL_{\BZ}(\infty)$, then
$x_{[i;a]} \in \SL_{\BZ}(\infty)$. 
Let $\CB$ be a basis of $\BU_q^-$. For $i \in I$ and $n \in \BN$,   
set $\CB_{i;n} = \{ b \in \CB \mid \ve_i(b) = n\}$. Thus we have a partition 
$\CB = \bigsqcup_{n \ge 0}\CB_{i;n}$. 
\par
We consider a basis $\bB$ of $\BU_q^-$ having the following properties;
\par\medskip\noindent
(C1) \ $\bB$ gives a $\BZ[q]$-basis of $\SL_{\BZ}(\infty)$, and an $\BA$-basis 
of ${}_{\BA}\BU_q^-$.  
\\
(C2) \ $\bB$ is bar-invariant, namely, $\ol b = b$ for any $b \in \bB$. 
\\
(C3) \ $\bB$ is almost orthonormal.
\\
(C4) \ For $\nu \in Q_-$, set $\bB_{\nu} = \bB \cap (\BU_q^-)_{\nu}$. 
Then we have a partition $\bB = \bigsqcup_{\nu \in Q_-}\bB_{\nu}$, where
$\bB_{\nu} = \{ 1\}$ if $\nu = 0$. 
\\
(C5) \ If $b \in \bB_{i;a}$ for $i \in I, a\ge 0$, then 
\begin{equation*}
\tag{1.12.1}
b \equiv b_{[i;a]} \mod q\SL_{\BZ}(\infty).
\end{equation*}
\\
(C6) \ $\bigcap_{i \in I}\bB_{i;0} = \{ 1 \}$.  
\\
(C7) \ 
Assume that $b \in \bB_{i;0}$. Then for any $a > 0$, 
there exists a unique element $b' \in \bB_{i, a}$ such that 
\begin{equation*}
\tag{1.12.2}
b' \equiv f_i^{(a)}b \mod f_i^{a+1}\BU_q^-.
\end{equation*}
The correspondence $b \mapsto b'$ gives a bijection 
$\pi_{i,a} : \bB_{i;0} \isom \bB_{i;a}$. 
\par\medskip
$\bB$ is called the {\it canonical basis} of $\BU_q^-$.  
The word {\it canonical} is justified by the following lemma. 

\begin{lem}  
The basis $\bB$ of $\BU_q^-$ is unique if it exists. 
\end{lem}

\begin{proof}
Let $\bB$ be a canonical basis of $\BU_q^-$. 
Assume that $\bB' = \bigsqcup_{\nu}\bB'_{\nu}$ is a basis 
satisfying similar properties. 
We show, by induction on $|\nu|$, that $\bB_{\nu} = \bB'_{\nu}$.
If $\nu = 0$, this holds by (C4). 
Assume that $\nu \ne 0$ and that $\bB_{\nu'} = \bB'_{\nu'}$ for 
any $|\nu'| < |\nu|$. Take $b \in \bB'_{\nu}$. By (C6), there exists $i \in I$
such that 
$\ve_i(b) = a > 0$.  Then there exists $b' \in \bB'_{i;0}$
such that $b \equiv f_i^{(a)}b'$ mod $f_i^{a+1}\BU_q^-$ 
by (C7).  
$b_{[i;a]} = (f_i^{(a)}b')_{[i;a]}$,  and 
$b \equiv (f_i^{(a)}b')_{[i;a]} \mod q\SL_{\BZ}(\infty)$ by (C5). 
By induction, $b' \in \bB_{i;0}$. Thus by applying (1.12.2) for 
$\bB$, there exists $b_1 \in \bB_{\nu}$ such that 
$b_1 \equiv f_i^{(a)}b' \mod f_i^{a+1}\BU_q^-$. 
By a similar discussion as above, we see that 
$b_1 \equiv (f_i^{(a)}b')_{[i;a]} \mod q\SL_{\BZ}(\infty)$. 
Thus $b - b_1$ is written as $b - b_1 = \sum_{b' \in \bB}a_{b'}b'$ with 
$a_{b'} \in q\BZ[q]$. 
Since $b - b_1$ is bar-invariant, this implies that $a_{b'} = 0$ for any $b'$, 
and so $b = b_1 \in \bB_{\nu}$. 
Hence  $\bB'_{\nu} \subset \bB_{\nu}$.  The opposite inclusion is obtained 
similarly, and we have $\bB'_{\nu} = \bB_{\nu}$.  The lemma is proved.  
\end{proof}    

\para{1.14.}
Using the decomposition $\BU_q^- = \bigoplus_{n \ge 0}f_i^{(n)}\BU_q^-[i]$
in Lemma 1.8, we write $x = \sum_{n \ge 0}f_i^{(n)}x_n$ with $x_n \in \BU_q^-[i]$.
We define $E'_i, F'_i : \BU_q^- \to \BU_q^-$ by 
\begin{equation*}
F'_i(x) = \sum_{n \ge 0}f_i^{(n+1)}x_n, \qquad 
E'_i(x) = \sum_{n \ge 1}f_i^{(n-1)}x_n.
\end{equation*}
$E'_i, F'_i$ are called Kashiwara operators on $\BU_q^-$. 

\para{1.15.}
Assuming the existence of the canonical basis $\bB$ for $\BU_q^-$, we define a map 
$F_i : \bB_{i;a} \to \bB_{i; a+1}$  
by 
\begin{equation*}
\xymatrix{
  F_i : \bB_{i; a} \ar[rr]^{\pi_{i;a}\iv} &   &  
            \bB_{i;0} \ar[rr]^{\pi_{i;a+1}} &  &  \bB_{i; a+1}. 
}
\end{equation*}
We define $E_i : \bB_{i;a} \to \bB_{i;a-1}$
as the inverse of $F_i$, if $a > 0$, and by $E_i(b) = 0$ if $a = 0$. 
The maps $E_i,F_i : \bB \to \bB \cup \{ 0\}$ are called  
Kashiwara operators on $\bB$.  
\par
Note that $\bB$ is an adapted basis of $\BU_q^-$ in the sense of 
[L2, 16.3.1].  Hence by [L2, Lemma 16.3.3, Prop. 16.3.5], we have the following.
\begin{prop} 
Let $b \in \bB_{i; a}$ for $i \in I$ and $a \ge 0$.  
\begin{enumerate}
\item \ $F'_i(b) \in \bB_{i; a+1} +  q\SL_{\BZ}(\infty)$. 
$E'_i(b) \in \bB_{i; a-1} +  q\SL_{\BZ}(\infty)$ if $a > 0$
and $E_i'(b) \equiv 0 \mod q\SL_{\BZ}(\infty)$ if $a = 0$.
\item \ 
$F_i'(b) \equiv F_i(b) \mod q\SL_{\BZ}(\infty)$.   
$E_i'(b) \equiv E_i(b) \mod q\SL_{\BZ}(\infty)$ if $a \ge 1$.  
\end{enumerate}
\end{prop}
The following result was proved by Lusztig.

\begin{thm}  
Assume that $X$ is symmetric.  Then $\BU_q^-$ has the canonical basis $\bB$. 
The basis $\bB$ has the property that $*(\bB) = \bB$.
\end{thm}

\begin{proof}
In [L1], [L2], an $\BA$-basis $\bB$ of ${}_{\BA}\BU_q^-$ was constructed 
by using the geometry of quivers associated to $X$, 
which satisfies the properties (C2), (C3), (C4), and (C6), (C7). 
By [L2, Lemma 16.2.5], $\bB$ is a $\BZ[q]$-basis of $\SL_{\BZ}(\infty)$, 
hence (C1) holds.
We show that $\bB$ satisfies (C5). Take $b \in \bB_{i;a}$, and write it 
as $b = \sum_{n \ge a}f_i^{(n)}x_n$ with $x_n \in \BU_q^-[i]$. 
Since $(b,b) \in 1 + q\BA_0$, by Lemma 1.11 (i), there exists a unique 
$a_0$ such that $f_i^{(n)}x_n \in q\SL_{\BZ}(\infty)$ for all $n \ne a_0$.  
By Proposition 1.16, $E_i'^{a}b \equiv E_i^{a}b \mod q\SL_{\BZ}(\infty)$. 
Since $E_i^{a}b \in \bB_{i:0}$, $E_i'^{a+1}b \in q\SL_{\BZ}(\infty)$.
This implies that $a_0 = a$.  Hence 
$b \equiv f_i^{(a)}x_a = b_{[i;a]} \mod q\SL_{\BZ}(\infty)$, and (C5) holds.   
(This property is also proved in [L3, Prop. 1.8].)
The property $*(\bB) = \bB$ was also proved in [L1], [L2], by using the geometric 
method. 
\end{proof}

\para{1.18.}
We define a subset $\wt\CB$ of $\BU_q^-$ by 
\begin{equation*}
\tag{1.18.1}
\wt\CB = \{ x \in {}_{\BA}\BU_q^- \mid  \ol x = x, (x, x) \in 1 + q\BZ[[q]]\}. 
\end{equation*}  
If there exists a basis $\CB$ of $\BU_q^-$ such that 
$\wt\CB = \CB \sqcup -\CB$, then $\wt\CB$ is called the canonical signed basis.  
\par
In the case where $\BU_q^-$ has the canonical basis $\bB$, then we have 
\begin{equation*}
\tag{1.18.2}
\wt\CB = \bB \sqcup -\bB, 
\end{equation*}
hence $\wt\CB$ is the canonical signed basis. 
\par\medskip
In fact, since $\bB$ is almost orthonormal, $\bB \sqcup -\bB \subset \wt\CB$. 
Conversely, take $x \in \wt\CB$. By [L2, Lemma 16.2.5], there exists $b \in \bB$ such that
$x \equiv \pm b \mod q\SL_{\BZ}(\infty)$. Since $x$  and $b$ are bar-invariant, 
this implies that $x = \pm b \in \bB \sqcup -\bB$.  Hence (1.18.2) holds. 
\par
The following result is immediate from (C6).

\begin{prop}  
Assume that the canonical basis $\bB$ exists for $\BU_q^-$.
Then for any $b \in \bB$, there exists a sequence $i_1, \dots, i_N \in I$, and 
$c_1, c_2, \dots, c_N \in \BZ_{> 0}$ such that 
$b = F_{i_1}^{c_1}F_{i_2}^{c_2}\cdots F_{i_N}^{c_{i_N}}1$.
\end{prop}

\para{1.20.}
We review the theory of the modified quantized enveloping algebra $\dot \BU_q$ 
(see [L2, Chap.23]).  Here we use a slightly different formulation.  
Consider the $\BQ(q)$-vector space 
$\dot \BU_q = \bigoplus_{\la \in P}\BU_q^-\otimes \BU_q^+\otimes \BQ(q)a_{\la}$,
where $P$ is the weight lattice for $\BU_q$ and $\BQ(q)a_{\la}$ is a vector 
space generated by a vector $a_{\la}$ corresponding 
to $\la \in P$. We define a product on $\dot\BU_q$ by the following rule;
\begin{equation*}
\tag{1.20.1}
\begin{aligned}
a_{\la}a_{\mu} =\d_{\la\mu}&a_{\la}, \quad e_ia_{\la} = a_{\la + \a_i}e_i,
                  \quad f_ia_{\la} = a_{\la -\a_i}f_i, \\ 
                 &(e_if_j - f_je_i)a_{\la} = \d_{ij}[\,\lp h_i,\la\rp\, ]_{d_i}a_{\la}.
\end{aligned}
\end{equation*}
Since $1 = \sum_{\la \in P}a_{\la}$ is not contained in $\dot \BU_q$, 
the identity element does not exist in $\dot \BU_q$.  
\par
A $\dot\BU_q$-module $M$ is said to be unital ([L2, 23.1.4]) if 
\par\medskip
(i) \ for any $m \in M$, we have $a_{\la}m = 0$ for all but finitely many $\la \in P$, 
\par
(ii) \ for any $m \in M$, we have $\sum_{\la \in P}a_{\la}m = m$. 
\par\medskip 
Giving a unital $\dot\BU_q$-module is equivalent to giving 
a $\BU_q$-module with weight space decomposition. 

\para{1.21.}
We define an $\BA$-submodule ${}_{\BA}\dot\BU_q$ of $\dot\BU_q$ by  
\begin{equation*}
\tag{1.21.1}
{}_{\BA}\dot\BU_q = \bigoplus_{\la \in P}
              {}_{\BA}\BU_q^-\otimes {}_{\BA}\BU_q^+ \otimes \BA a_{\la}.
\end{equation*}
\par
It is known by ([L2, Lemma 23.2.2]) that ${}_{\BA}\dot\BU_q$
is an $\BA$-subalgebra of $\dot\BU_q$ generated by $e_i^{(n)}a_{\la}, f_i^{(n)}a_{\la}$
for various $i \in I, n \ge 0, \la \in P$. 
\par
Let $M(\la)$ be the Verma module for $\BU_q$ 
associated to $\la \in P$, with highest weight 
vector $v_{\la}$.  
Then $M(\la) \simeq \BU_q^-v_{\la} \simeq \BU_q^-$ as $\BU_q^-$-modules.
Let ${}_{\BA}M(\la) = {}_{\BA}\BU_q^-v_{\la}$ be the $\BA$-submodule of 
$M(\la)$. Note that $M(\la)$ is a unital $\dot\BU_q$-module, and it can be
verified ([L2, 23.3.2]) that ${}_{\BA}M(\la)$ is stable under the actions of 
$f_i^{(n)}a_{\la'}, e_i^{(n)}a_{\la'}$.  Hence ${}_{\BA}M(\la)$ is 
an ${}_{\BA}\dot\BU_q$-submodule of $M(\la)$ generated by $v_{\la}$. 

\para{1.22.}
Let $R$ be a commutative ring with 1, and with an invertible element $\Bq \in R$.
We fix a ring homomorphism $\f : \BA \to R$ such that $\f(q^n) = \Bq^n$ 
for any $n \in \BZ$. We regard $R$ as an $\BA$-algebra via $\f$, and 
consider
\begin{equation*}
{}_R\BU_q^{\pm} = R\otimes_{\BA}\BU_q^{\pm}, \qquad 
{}_R\dot\BU_q = R \otimes_{\BA}\dot\BU_q.
\end{equation*}

We have a direct sum decomposition 
${}_R\BU_q^{\pm} = \bigoplus_{\nu \in Q}({}_R\BU_q^{\pm})_{\nu}$, and 
${}_R\dot\BU_q$ is expressed as 
\begin{equation*}
{}_R\dot\BU_q = \bigoplus_{\la \in P}{}_R\BU_q^-\otimes {}_R\BU_q^+\otimes Ra_{\la}
\end{equation*}

\par
Unital ${}_R\dot\BU_q$-modules are defined similarly to 1.20.  Assume that $M$ is 
a unital ${}_R\dot\BU_q$-module.  Then $M = \bigoplus_{\la \in P}M_{\la}$, 
where $M_{\la} = a_{\la}M$, and $M_{\la}$ becomes an $R$-module. 
Following [L2, 31.3], we introduce the notion of highest weight modules.
Let $M$ be a unital ${}_R\dot\BU_q$-module. $M$ is called a highest weight
module with highest weight $\la \in P$ if there exists a vector $m \in M_{\la}$
such that 
\par\medskip
\begin{enumerate}
\item \ $e_i^{(n)}m = 0$ for any $i \in I$ and $n > 0$.
\item \ $M = {}_R\BU_q^-m$.
\item \ $M_{\la}$ is a free $R$-module of rank 1 with generator $m$.
\end{enumerate}

\par
A unital ${}_R\dot \BU_q$-module $M$ is said to be integrable 
if for any $m \in M$ and any $i \in I$, there exists $n_0 \ge 1$ such that
$e_i^{(n)}m = f_i^{(n)}m = 0$ for all $n \ge n_0$. 
 
\para{1.23.}
For $\la \in P$, let ${}_{\BA}M(\la)$ be the Verma module defined in 1.21.
Set ${}_RM(\la) = R\otimes_{\BA}{}_{\BA}M(\la)$.  Then ${}_RM(\la)$ is a highest 
weight ${}_R\dot\BU_q$-module with highest weight $\la$. 
Since ${}_{\BA}M(\la) \simeq {}_{\BA}\BU_q^-$ as $\BA$-modules, we have
${}_RM(\la) \simeq {}_R\BU_q^-$ as $R$-modules. 
\par
Let $P^+$ be the set of dominant weights $\la$  in $P$, namely, $\la$ such that
$\lp h_i, \la \rp \ge 0$ for any $i \in I$.   
For $\la \in P^+$, let $L(\la)$ be the integrable highest weight $\BU_q$-module
with highest weight $\la$ and highest weight vector $v_{\la}$  
As an $\BU_q^-$-module, it is written as 
\begin{equation*}
\tag{1.23.1}
L(\la) = \BU_q^-/\sum_{i, n > \, \lp h_i, \la\rp}\BU_q^-f_i^{(n)}.
\end{equation*} 
Let ${}_{\BA}L(\la)$ be the ${}_{\BA}\BU_q^-$-submodule of $L(\la)$ 
generated by $v_{\la}$. 
By using the commuting relation in [L2, 31.1.6], one can check 
that the actions of $f_i^{(n')}a_{\la'}, e_i^{(n')}a_{\la'}$ 
($i \in I, n' \in \BN, \la' \in P$) on ${}_{\BA}\BU_q^-$
preserve $\sum_{i, n > \, \lp h_i,\la\rp}{}_{\BA}\BU_q^-f_i^{(n)}$. 
Hence ${}_{\BA}L(\la)$ is a unital ${}_{\BA}\dot\BU_q$-module.
We define a unital ${}_R\dot \BU_q$-module ${}_RL(\la)$ by 
${}_RL(\la) = R \otimes_{\BA}{}_{\BA}L(\la)$.  
Then ${}_RL(\la)$ is a highest weight
module with highest weight $\la$, and is integrable. 
${}_RL(\la)$ is a quotient of ${}_RM(\la)$. 
\par
Assume that $R$ is any field such that $\Bq_i = \f(q_i)$ is not a root of
unity. It is proved by Tanisaki [T, Thm. 5.5, Thm. 5.6] that 
the Weyl-Kac type character formula 
holds for ${}_R\dot\BU_q$, and in particular, any integrable highest weight 
${}_R\dot\BU_q$-module is irreducible.  By applying this to our situation, we have
\par\medskip\noindent
(1.23.2) \ Let $R$ be as above.  Then the integrable highest weight module
${}_RL(\la)$ is irreducible. 

\para{1.24.}
Return to the case where $R$ is a commutative ring. 
We consider the map $r : {}_{\BA}\BU_q^- \to {}_{\BA}\BU_q^- \otimes_{\BA}{}_{\BA}\BU_q^-$,
and ${}_ir, r_i : {}_{\BA}\BU_q^- \to {}_{\BA}\BU_q^-$ as in 1.6. 
By tensoring with $R$, one can define maps 
$r: {}_R\BU_q^- \to {}_R\BU_q^-\otimes_R{}_R\BU_q^-$ and 
${}_ir, r_i : {}_R\BU_q^- \to {}_R\BU_q^-$.  Then ${}_ir, r_i$ also satisfy 
similar properties as in (1.3.6). The following result is an ${}_R\dot\BU_q$-version of 
[L2, Prop. 3.1.6], and can be proved in a similar way by using (1.3.6).

\begin{lem}  
Assume that $\Bq_i = \f(q_i) \ne \pm 1$.  
Then for any $x \in {}_R\BU_q^-, i \in I$, and $\la \in P$, 
the following relation holds in ${}_R\dot \BU_q$. 
\begin{equation*}
\tag{1.25.1}
(e_ix - xe_i)a_{\la} = 
\frac{\Bq_i^{\lp h_i,\la\rp }r_i(x) - \Bq_i^{-\lp h_i,\la + \weit x + \a_i\rp }{}_ir(x)}
                             {\Bq_i - \Bq_i\iv}a_{\la}.
\end{equation*}
\end{lem}

By making use of Lemma 1.25, we prove the following formula, which 
is a generalization of Lemma 1.5. 

\begin{prop}  
Assume that $R$ is a field such that $\Bq_i$ is not a root of unity. 
Consider maps ${}_ir, r_i : {}_R\BU_q^- \to {}_R\BU_q^-$.
Then we have
\begin{equation*}
\tag{1.26.1}
\bigcap_{i \in I}\Ker {}_ir = R\cdot 1, \qquad \bigcap_{i \in I}\Ker r_i = R \cdot 1.
\end{equation*}
\end{prop}

\begin{proof}
We show the first formula. The second one  
follows from  it by applying $*$.
It is enough to show that $\bigcap_{i \in I}\Ker {}_ir \cap ({}_R\BU_q^-)_{\nu} = 0$
for any weight $\nu \ne 0$.  We prove this by induction on $|\nu|$.   
Take $x \in \bigcap_{i \in I}\Ker {}_ir \cap ({}_R\BU_q^-)_{\nu}$. 
If $|\nu| = 1$, then $x = cf_i$ for some $i$,  and $c \in R$.  
In this case, ${}_ir(x) = c = 0$, and $x = 0$.  
Assume that $|\nu| \ge 2$. By a similar discussion as in the proof of 
Lemma 1.4, it is shown  
that ${}_ir$ and $r_j$ commute each other for any $i,j$.  
Thus $r_j(x) \in \bigcap_{i \in I}\Ker {}_ir$. 
Hence by induction, $r_j(x) = 0$ for any $j$. 
We apply Lemma 1.25 for this $x$. Then $(e_ix - xe_i)a_{\la} = 0$ 
in ${}_R\dot\BU_q$ for any $i \in I, \la \in P$. We consider the highest
weight ${}_R\dot\BU_q$-module $M = {}_RL(\la)$ with highest weight vector $v_{\la}$ 
as in (1.23.1).  
Then 
\begin{equation*}
\tag{1.26.2}
(e_ix - xe_i)a_{\la}v_{\la} = (e_ix - xe_i)v_{\la} = e_ixv_{\la} = 0.
\end{equation*}

$xv_{\la}$ is contained in the weight space $({}_RL(\la))_{\la'}$ with 
weight $\la' \ne \la$. and $e_i(xv_{\la}) = 0$ for any $i$.  Since $[n]_i \ne 0$ in $R$, 
we also have $e_i^{(n)}(xv_{\la}) = 0$ for any $n \ge 1$. 
Thus $xv_{\la}$ generates a proper ${}_R\dot\BU_q$-submodule if $xv_{\la} \ne 0$.  
Since ${}_RL(\la)$ is irreducible by (1.23.2),  we have $xv_{\la} = 0$. 
This is true for any $\la \in P$ such that $\lp h_i, \la \rp \ge 0$ for any $i \in I$.  
But by the expression of ${}_RL(\la)$ in (1.23.1), if $x \ne 0$, one can find
some $\la \in P$ such that $xv_{\la} \ne 0$. This implies that $x = 0$. 
The proposition is proved.
\end{proof}

\section{ The diagram automorphism }

\para{2.1.}
Let $X = (I, (\ ,\ ))$ be a Cartan datum, and let $\s: I \to I$ 
be a permutation such that $(\a_{\s(i)}, \a_{\s(j)}) = (\a_i, \a_j)$
for any $i,j \in I$.  
Such a $\s$ is called a diagram automorphism on $X$. 
Let $\ul I$ be the set of $\s$-orbits in $I$. 
$\s$ is called admissible if for any orbit $\eta \in \ul I$, 
$(\a_i, \a_j) = 0$ for any $i \ne j \in \eta$. Hereafter, we assume 
that $\s$ is an admissible diagram automorphism on $X$. 
\par
Let $\ul I$ be as above. We define a symmetric bilinear form  $(\ ,\ )_1$ on 
$\bigoplus_{\eta \in \ul I}\BQ \a_{\eta}$ by 
\begin{equation*}
\tag{2.1.1}
(\a_{\eta}, \a_{\eta'})_1 = \begin{cases}
                            (\a_i, \a_i) |\eta|,  \qquad (i \in \eta)
                             &\quad\text{ if } \eta = \eta', \\
                            \sum_{i \in \eta, j \in \eta'}(\a_i,\a_j)
                             &\quad\text{ if }  \eta \ne \eta'.
                         \end{cases} 
\end{equation*}
Then $(\a_{\eta}, \a_{\eta})_1 \in 2\BZ_{> 0}$ for each $\eta \in \ul I$. 
Since 
\begin{align*}
\sum_{i \in \eta, j \in \eta'}(\a_i, \a_j)  
   = |\eta|\sum_{j \in \eta'}(\a_i, \a_j) \qquad \text{ for $i \in \eta$ }, 
\end{align*}
we have, for a fixed $i \in \eta$,  
\begin{align*}
\tag{2.1.2}
a_{\eta \eta'} = \frac{2(\a_{\eta}, \a_{\eta'})_1}{(\a_{\eta}, \a_{\eta})_1}
      = \sum_{j \in \eta'}\frac{2(\a_i,\a_j)}{(\a_i,\a_i)} 
      = \sum_{j \in \eta'}a_{ij} \in \BZ_{\le 0}.
\end{align*}
Hence $(\ ,\ )_1$ satisfies (i), (ii) in 1.1, and   
$\ul X = (\ul I, (\ ,\ )_1)$ is a Cartan datum.
The Cartan datum $\ul X$ is called the Cartan datum induced from $(X, \s)$. 
\par
Note that a symmetric Cartan datum $X$ is naturally identified with 
a finite graph $\vG$ with multiple edges (for a given $X$, $I$ is the set 
of vertices in $\vG$, 
$i$ and $j$ are joined by $-(\a_i, \a_j)$ edges for $i \ne j \in I$).
The diagram automorphism on $X$ corresponds to the graph automorphism.  
In the case where $X$ is symmetric, the definition of $(\ ,\ )_1$ 
coincides with the one given in [L2, 14.1.1], and in [SZ2, 2.1]. 

\para{2.2.}
It is known by [L2, Prop.14.1.2] that, for a given Cartan datum 
$X$, there exists a symmetric Cartan datum $\wt X$, and 
an admissible diagram automorphism $\s$ on $\wt X$ such that the  
Cartan datum induced from $(\wt X, \s)$ is isomorphic to $X$.
$\wt X$ is constructed as follows;  let $X = (I, (\ ,\ ))$ be a Cartan datum, 
where $(\ ,\ )$ is the bilinear form  on the vector space 
$\bigoplus_{i \in I}\BQ\a_i$. Recall that $d_i = (\a_i,\a_i)/2$ for each $i \in I$.
We consider a set $D_i$ of cardinality $d_i$, and let $\s : D_i \to D_i$ be 
a cyclic permutation on $D_i$. We consider a set $D = \bigsqcup_{i \in I}D_i$, 
and a permutation $\s$ on $D$ induced from $\s : D_i \to D_i$.
Let $\ul D$ be the set of $\s$-orbits on $D$.  Then $\ul D$ is naturally 
in bijection with $I$.  We shall define a graph with vertex set $D$.      
Fix $i \ne j \in I$. Since $(\a_i, \a_j)$ is divisible by $d_i$ and $d_j$, 
it is divisible by the smallest common multiple $l(d_i, d_j)$ of $d_i$ and $d_j$.
Choose $x \in D_i, y \in D_j$, and join $x$ and $y$ by $c$-fold edges, where 
$c = -(\a_i,\a_j)/l(d_i,d_j)$. If $(x',y') \in D_i \times D_j$ is 
$(\s \times \s)$-conjugate to $(x, y)$, we join $x'$ and $y'$ by the same number of 
edges. Hence the number of edges joining $D_i$ and $D_j$ is equal to $-(\a_i,\a_j)$, 
which is independent from the choice of $x, y$. 
We define a graph $\vG = (D, \Om)$, where $\Om$ is the set of edges defined  above.
Let $\wt X = (D, (\ ,\ )_0)$ be the Cartan datum  obtained from 
$\vG$, where 
$(\ ,\ )_0$ is the bilinear form on $\bigoplus_{x \in D}\BQ \a_x$ such that 
\begin{equation*}
(\a_x, \a_y)_0 = \begin{cases}
                        2  &\quad\text{ if $x = y$},  \\
                        -\sharp\{ \text{edges joining $x$ and $y$ } \}
                           &\quad\text{ if $x \ne y$.}
                      \end{cases}
\end{equation*}
Then $\wt X$ is a Cartan datum of symmetric type, 
and $\s: D \to D$ gives an admissible 
 diagram automorphism on $\wt X$. The Cartan datum 
$(\ul D, (\ ,\ )_1)$ induced from $(\wt X, \s)$ is isomorphic to $X$. 

\para{2.3.}
Let $X = (I, (\ ,\ ))$ be a Cartan datum, 
with an admissible automorphism $\s$.
Assume that the order of $\s$ is $n = am$, where $a$ and $m$ are prime each other.
Then $\s$ is decomposed as $\s = \tau\tau'$, where $\tau, \tau'$ are powers of $\s$, 
and the order of $\tau$ (resp. $\tau'$) is equal to
$a$ (resp. $m$).  
Let $\ul I$ be the set of $\s$-orbits
in $I$ as before. Let $I^{\tau}$ be the set of $\tau$-orbits in $I$. 
Then $\s$ permutes $I^{\tau}$, whose action we denote by $\ul\s$.  
This action coincides with the induced action of $\tau'$ on $I^{\tau}$, and so
the order of $\ul\s$ is equal to $m$. The set of $\ul\s$-orbits in $I^{\tau}$
is naturally identified with $\ul I$. 
\par
The Cartan datum $\ul X = (\ul I, (\ ,\ )_1)$ induced from $(X, \s)$ 
is defined as in (2.1.1), 
namely   
\begin{equation*}
\tag{2.3.1}
(\a_{\eta}, \a_{\eta'})_1 = \begin{cases}
                               (\a_i,\a_i)|\eta|, \quad (i \in \eta)  
                                 &\quad\text{ if $\eta = \eta'$, } \\
                               \sum_{ i \in \eta, j \in \eta'} (\a_i, \a_j)
                                        &\quad\text{ if $\eta \ne \eta'$. }
                           \end{cases} 
\end{equation*}
\par
Let $X^{\tau} = (I^{\tau}, (\ ,\ )_{\tau})$ be the Cartan datum induced from 
$(X, \tau)$.  Then for $\g, \g' \in I^{\tau}$ we have, again by (2.1.1),   
\begin{equation*}
\tag{2.3.2}
(\a_{\g}, \a_{\g'})_{\tau} = \begin{cases}
                               (\a_i,\a_i)|\g|, \quad (i \in \g)  
                           &\quad\text{ if $\g = \g'$, } \\
                               \sum_{i \in \g, j \in \g'}(\a_i, \a_j)  
                              &\quad\text{ if $\g \ne \g'$. }
                           \end{cases} 
\end{equation*}
\par
On the other hand, we 
consider the Cartan datum $X^{\tau}$ and the admissible automorphism 
$\ul\s$ on $X^{\tau}$.  We want to show that the Cartan datum  
$\ul X^{\tau} = (\ul I^{\tau}, (\ ,\ )_2)$ induced from $(X^{\tau}, \ul\s)$ 
is  canonically isomorphic to $\ul X$.  Note that $\ul I^{\tau}$ is in bijection 
with $\ul I$, which we denote by $\wt\eta \lra \eta$.  
We consider the symmetric bilinear form $(\ ,\ )_2$ on the vector space 
$\bigoplus_{\wt\eta \in \ul I^{\tau}}\BQ \a_{\wt \eta}$. 
By (2.1.1), $(\ ,\ )_2$ is defined by 

\begin{equation*}
\tag{2.3.3}
(\a_{\wt\eta}, \a_{\wt\eta'})_2 = 
             \begin{cases}
       (\a_{\g}, \a_{\g})_{\tau}|\wt\eta|,  \qquad (\g \in \wt\eta)
              &\quad\text{ if $\wt\eta = \wt\eta'$, }  \\
             \sum_{\g \in \wt\eta, \g' \in \wt\eta'}(\a_{\g}, \a_{\g'})_{\tau}
              &\quad\text{ if $\wt\eta \ne \wt\eta'$.}
             \end{cases}
\end{equation*}
Let $\wt\eta$ be the $\ul\s$-orbit of $\g \in I^{\tau}$. 
Then $|\eta| = |\wt\eta||\g|$.  Since 
$(\a_{\g}, \a_{\g})_{\tau} = (\a_i,\a_i)|\g|$ for $i \in \g$, we obtain
\begin{equation*}
(\a_{\wt\eta}, \a_{\wt\eta})_2 = (\a_{\eta}, \a_{\eta})_1.
\end{equation*}  
Moreover, for $\wt\eta \ne \wt\eta'$, 
\begin{align*}
(\a_{\wt\eta}, \a_{\wt\eta'})_2 = \sum_{\g \in \wt\eta, \g'\in \wt\eta'}
                    \sum_{i \in \g, j \in \g'}(\a_i, \a_j)
             = \sum_{i \in \eta, j \in \eta'}(\a_i, \a_j) = (\a_{\eta}, \a_{\eta'})_1.
\end{align*}
Thus under the identification $\bigoplus_{\eta \in \ul I}\BQ \a_{\eta}
      \simeq \bigoplus_{\wt\eta \in \ul I^{\tau}}\BQ\a_{\wt\eta}$,
$(\ ,\ )_1$ coincides with $(\ ,\ )_2$.   
Summing up the above discussion, we have

\begin{lem}  
The Cartan datum $(\ul I^{\tau}, (\ ,\ )_2)$ induced from $(X^{\tau}, \ul\s)$ 
is isomorphic to the Cartan datum $\ul X = (\ul I, (\ ,\ )_1)$ induced from 
$(X, \s)$. 
\end{lem} 

\begin{prop}  
Let $X$ be a Cartan datum. Let $\wt X$ be a Cartan datum of 
symmetric type with admissible $\s : \wt X \to \wt X$ such that 
the Cartan datum induced  
from $(\wt X, \s)$ is isomorphic to $X$ as given in 2.2.
Then there exists a sequence $\wt X = X_0, X_1, \dots, X_k = X$ of 
Cartan data, and an
admissible diagram automorphism $\s_i : X_i \to X_i$ such that the 
Cartan datum  induced from $(X_i, \s_i)$ is isomorphic to $X_{i+1}$, where the order
of $\s_i$ is a prime power, and that $\s = \s_{k-1}\cdots \s_1\s_0$. 
\end{prop}

\begin{proof}
Let $n$ be the order of $\s$.  We prove the proposition by induction on $n$.
If $n$ is a prime power, there is nothing to prove.  Assume that $n$ is not
a prime power. We write $n = am$, where $m$ is a power of some prime number $p$, 
and $a$ is prime to $p$. Write $\s = \tau\tau' = \tau'\tau$ as in 2.3, where 
the order of $\tau$ (resp. $\tau'$) is equal to $a$ (resp. $m$).  
Then $\tau$ is an admissible 
automorphism on $\wt X$, and we denote by $X'$ the Cartan datum induced   
from $(\wt X, \tau)$. $\s$ induces an admissible automorphism on $X'$, which 
we denote by $\s'$. Here the order of $\s'$ is equal to $m$, and the order of 
$\tau$ is equal to $a$.  By Lemma 2.4, the Cartan datum induced from $(X', \s')$ 
is isomorphic to $X$. By applying the induction hypothesis on $\wt X, \tau$ and $X'$, 
we obtain a sequence $\wt X = X_0, X_1, \dots, X_{k-1} = X'$,
 and $\s_i: X_i \to X_i$ such that $\tau = \s_{k-2}\cdots \s_0$ 
satisfying the condition.  By setting $\s_{k-1} = \s'$, 
we obtain the proposition.  
\end{proof}

\par\bigskip
\section{The algebra $\BV_q$ } 

\para{3.1.}
Let $X = (I, (\ ,\ ))$ be a Cartan datum, and 
$\s$ an admissible diagram automorphism on $X$.
Let $\BU_q^-$ be the quantized enveloping algebra associated to 
$X$.  Then $\s$ induces an algebra automorphism 
$\s : \BU_q^- \to \BU_q^-$ by $f_i \mapsto f_{\s(i)}$.
Let $\BU_q^{-,\s}$ be the subalgebra of $\BU_q^-$ consisting of 
$\s$-fixed elements.  If the canonical basis $\bB$ exists for $\BU_q^-$, 
then $\s(\bB)$ is also the canonical basis. Hence by the uniqueness 
property (Lemma 1.13), $\s(\bB) = \bB$.  $\s$ acts on $\bB$ as a permutation, 
and we denote by $\bB^{\s}$ the set of $\s$-fixed elements in $\bB$.  
Let $\ul X = (\ul I, (\ ,\ )_1)$ 
be the Cartan datum induced from $(X, \s)$, and $\ul\BU_q^-$ 
the associated quantized algebra. We will compare the algebra structure 
of $\BU_q^{-,\s}$ and $\ul\BU_q^-$, as in [SZ2]. 

\para{3.2.}
From now on, throughout this section, we assume that 
the order of $\s$ is a power of a prime number $p$. 
We also assume that the canonical basis $\bB$ exists for $\BU_q^-$. 
(In the case where $p = 2$, we need to replace $\bB$ by the canonical 
signed basis $\wt\CB = \CB \sqcup -\CB$. However since 
all the discussion in this section works well even in the case $p = 2$  
under a suitable modification, we concentrate to the case where $\bB$ exists.)
\par
$\s$ stabilizes ${}_{\BA}\BU_q^-$, and we define 
${}_{\BA}\BU_q^{-,\s} = \BU_q^{-,\s} \cap {}_{\BA}\BU_q^-$, the subalgebra of 
${}_{\BA}\BU_q^-$ consisting of $\s$-fixed elements.
Let $\BF = \BZ/p\BZ$ be a finite field of $p$ elements, and set 
$\BA' = \BF[q, q\iv]$.
We consider the $\BA'$-algebra
\begin{equation*}
{}_{\BA'}\BU_q^{-,\s} = \BA'\otimes_{\BA}{}_{\BA}\BU_q^{-,\s}
                     \simeq {}_{\BA}\BU_q^{-,\s}/p({}_{\BA}\BU_q^{-,\s}). 
\end{equation*}  
For each $x \in \BU_q^-$, we denote by $O(x)$ the orbit sum of $x$, namely 
$O(x) = \sum_{0 \le i < k}\s^i(x)$, where $k$ is the smallest integer such that 
$\s^k(x) = x$. Hence $O(x)$ is $\s$-invariant.  
$O(x)$ is defined similarly for ${}_{\BA}\BU_q^-, {}_{\BA'}\BU_q^-$.
Let $J$ be an $\BA'$-submodule of ${}_{\BA'}\BU_q^{-, \s}$ generated by 
$O(x)$ for $x \in {}_{\BA'}\BU_q^-$ such that $\s(x) \ne x$.  
Then $J$ is a two-sided ideal of ${}_{\BA'}\BU_q^{-,\s}$. 
We define an $\BA'$-algebra $\BV_q$ as the quotient algebra 
${}_{\BA'}\BU_q^{-,\s}/J$.  
Let $\pi : {}_{\BA'}\BU_q^{-,\s} \to \BV_q$ be the natural projection.  
Note that 
$\BV_q$ is a generalization of the algebra 
$\BV_q$ introduced in [SZ1, SZ2]. 

\para{3.3.}
For each $\eta \in \ul I$ and $a \in \BN$, set 
$\wt f_{\eta}^{(a)} = \prod_{i \in \eta}f_i^{(a)}$. 
Since $f_i^{(a)}$ and $f_j^{(a)}$ commute each other for $i,j \in \eta$,  
we have $\wt f^{(a)}_{\eta} \in {}_{\BA}\BU_q^{-,\s}$. We denote its image 
in ${}_{\BA'}\BU_q^{-,\s}$ also by $\wt f^{(a)}_{\eta}$.  
Thus we can define $g_{\eta}^{(a)} \in \BV_q$ by 
\begin{equation*}
\tag{3.3.1}
g^{(a)}_{\eta} = \pi(\wt f_{\eta}^{(a)}). 
\end{equation*}
In the case where $a = 1$, we set 
$\wt f_{\eta}^{(1)} = \wt f_{\eta} = \prod_{i \in \eta}f_i$ 
and $g^{(1)}_{\eta} = g_{\eta}$. 
\par
Since $*$ commutes with $\s$, $*$ preserves ${}_{\BA}\BU_q^{-,\s}$,
and acts on ${}_{\BA'}\BU_q^{-\s}$, which induces an anti-algebra 
automorphism $*$ on $\BV_q$. Note that $\wt f_{\eta}^{(a)}$ is 
$*$-invariant since $f_i$ and $f_j$ commute each other for $i, j \in \eta$.
Thus $g_{\eta}^{(a)}$ is $*$-invariant. 
\par
Let $\ul\BU_q^-$ be as above. 
The algebras ${}_{\BA}\ul\BU_q^-$, and 
${}_{\BA'}\ul\BU_q^-$ are defined similarly to $\BU_q^-$. 
We denote by $\ul f_{\eta}$  ($\eta \in \ul I$) the generators of $\ul\BU_q^-$, 
and $\ul f^{(a)}_{\eta}$ ($\eta \in \ul I, a \in \BN$) the generators in 
${}_{\BA}\ul\BU_q^-$. 
The anti-algebra automorphism $*$ on ${}_{\BA'}\ul\BU_q^-$ is inherited from 
$*$ on $\ul\BU_q^-$. The following result is a generalization of Theorem 2.4 in [SZ2].

\begin{thm}   
Assume that $\BU_q^-$ has the canonical (signed) basis. 
The assignment  
$\ul f_{\eta}^{(a)} \mapsto g_{\eta}^{(a)}$ gives an isomorphism 
$\Phi : {}_{\BA'}\ul\BU_q^- \isom \BV_q$ of $\BA'$-algebras, which is compatible 
with $*$-operation. 
\end{thm}

First we note that 
\begin{prop}   
The assignment
$\ul f_{\eta}^{(a)} \mapsto g_{\eta}^{(a)}$ gives a homomorphism 
$\Phi : {}_{\BA'}\ul\BU_q^- \to \BV_q$. 
\end{prop}

The proof of Proposition 3.5 will be given in Section 5.
Note that the corresponding result in the finite or affine case was 
proved in [SZ2,  Prop. 2.6] by the computation using the PBW-bases.
This method cannot be applied to our case, since the PBW-bases do not
exist for Kac-Moody case. We prove Proposition 3.5, by a purely combinatorial
argument, which is in some sense simpler than the computation by PBW-bases. 

\para{3.6.}
Assuming Proposition 3.5, we continue the discussion.
Let $\wh\BA = \BZ((q)) \cap \BQ(q)$ be the subring of $\BQ(q)$ containing 
$\BA = \BZ[q, q\iv]$.
By (1.3.2), we have $(f_i, f_j) \in \wh \BA$. Since 
${}_ir$ preserves ${}_{\BA}\BU_q^-$, we have
$({}_{\BA}\BU_q^-, {}_{\BA}\BU_q^-) \subset \wh\BA$ by (1.3.5). 
Recall that $\BA' = \BF[q,q\iv]$, and 
${}_{\BA'}\BU_q^- = {}_{\BA}\BU_q^-/p({}_{\BA}\BU_q^-)$. 
Let $\BF(q)$ be the field of rational functions of $q$ with coefficients in $\BF$. 
The bilinear form  $(\ ,\ )$ on ${}_{\BA}\BU_q^-$ 
induces a bilinear form $(\ ,\ )$ on ${}_{\BA'}\BU_q^-$, whose 
values are in $\wh\BA/p\wh\BA \subset \BF((q)) \cap \BF(q) = \BF(q)$.
\par
Since $\s$ commutes with $r$, 
the bilinear form $(\ ,\ )$ on $\BU_q^-$ is $\s$-invariant, namely,
$(\s(x), \s(y)) = (x, y)$ for any $x, y \in \BU_q^-$.  
In particular, for $x, y \in {}_{\BA}\BU_q^-$ such that $\s(x) \ne x$, 
we have
\begin{align*}
\tag{3.6.1}
(O(x), O(y)) = \sum_{x' \in O(x)}(x', O(y)) = |O(x)|(x, O(y)) \in p\BZ((q)) \cap \BQ(q)
\end{align*}
since $|O(x)| > 1$ is a power of $p$. 
It follows that the bilinear form on ${}_{\BA}\BU_q^{-,\s}$ 
induces a bilinear form $(\ , \ )$ on ${}_{\BF(q)}\BV_q = \BF(q)\otimes_{\BA'}\BV_q$. 
Since $\s$ permutes $\bB$, $\pi(\bB^{\s})$ gives an 
$\BA'$-basis of $\BV_q$, hence gives a basis of ${}_{\BF(q)}\BV_q$. 
\par
On the other hand, the bilinear form $(\ ,\ )$ on ${}_{\BA'}\ul\BU_q^-$ is defined 
similarly, and it induces a bilinear form $(\ ,\ )$ on 
${}_{\BF(q)}\ul\BU_q^- = \BF(q)\otimes_{\BA}{}_{\BA}\ul\BU_q^-$ with 
values in $\BF(q)$.  Here we regard $\BF(q)$ as an $\BA$-algebra by a homomorphism 
$\f : \BA \to \BF(q), q^n \mapsto q^n$.  
by applying Proposition 1.26 for $R = \BF(q)$, we obtain the following. 
Note that the existence of the canonical basis is not necessary for 
the proof of this fact. 
\begin{prop}  
The bilinear form $(\ ,\ )$ on ${}_{\BF(q)}\ul\BU_q^-$ is non-degenerate.
\end{prop} 

\begin{proof}
Take $x \in ({}_{\BF(q)}\BU_q^-)_{\nu}$.  It is enough to show that  
if $(x, ({}_{\BF(q)}\BU_q^-)_{\nu}) = 0$, then $x = 0$.  We prove this by induction 
on $|\nu|$.  This is clearly true for $\nu = 0$.  We assume that $|\nu| \ge 1$, and 
that the claim holds for $\nu'$ such that $|\nu'| < |\nu|$. 
By (1.3.5), and by using $(f_i, f_i) \ne 0$, we see that $(y, {}_ir(x)) = 0$ 
for any $y$.  Thus, by induction, ${}_ir(x) = 0$. Since this is true for any $i \in I$, 
by Proposition 1.26, we conclude that $x = 0$.  The proposition is proved.  
\end{proof}

\begin{prop}  
Under the notation in 3.6, we have
\begin{enumerate}
\item  For any $x,y \in {}_{\BA'}\ul\BU_q^-$, $(\Phi(x), \Phi(y)) = (x,y)$. 
\item The map $\Phi : {}_{\BA'}\ul\BU_q^- \to \BV_q$ is injective. 
\end{enumerate}
\end{prop} 

\begin{proof}
The map $\Phi$ can be extended to the  map 
$\wt\Phi: {}_{\BF(q)}\ul\BU_q^- \to {}_{\BF(q)}\BV_q$.
To prove (ii), it is enough to show that $\wt\Phi$ is injective. 
This follows from (i) since the bilinear form on ${}_{\BF(q)}\ul\BU_q^-$
is non-degenerate by Proposition 3.7.  
Hence it is enough to prove (i).  The proof of (i) is done in an almost similar 
way as in the proof of [SZ2, Prop. 2.8].  
We just give some additional remarks. We use the same notation as in the proof 
of Proposition 2.8 in [SZ2]. The definition of 
$J_1$ given there should be replaced by an $\BA$-submodule of 
$({}_{\BA}\BU_q^-\otimes {}_{\BA}\BU_q^-)^{\s}$ generated by orbit sums $O(z)$ 
for $z \in {}_{\BA}\BU_q^-\otimes {}_{\BA}\BU_q^-$ such that $\s(z) \ne z$.
Let $\eta_1 \in \ul I$. 
The computation of $r(\wt f_{\eta_1})$ in  [SZ2, (2.8.2)]  is done as follows; 
Take $\eta_1 \in \ul I$. Then $f_i\otimes 1 + 1 \otimes f_i$ are commuting each other 
for $i \in \eta_1$, and we have
\begin{align*}
r(\wt f_{\eta_1}) &= \prod_{i \in \eta_1}(f_i\otimes 1 + 1 \otimes f_i)  \\
                  &= \sum_{\zeta \subset \eta_1} 
         \biggl(\prod_{i \in \zeta}f_i \otimes \prod_{j \in \eta_1 - \zeta}f_j\biggr) \\
                  &= \wt f_{\eta_1}\otimes 1 + 1 \otimes \wt f_{\eta_1} \mod J_1. 
\end{align*}
Since $\wt f_{\eta_1} \otimes 1 + 1 \otimes \wt f_{\eta_1} = \wt{\ul r(\ul f_{\eta_1})}$, 
(2.8.1) in [SZ2] holds for $k = 1$.  Then (2.8.2) is proved by a similar argument.  
For the proof of (2.8.7) in [SZ2], we use, for $z_1 \in \BU_q^-\otimes \BU_q^-$ 
such that $\s(z_1) \ne z_1$,
\begin{equation*}
(O(z_1), \wt f_{\eta_1'}\otimes \wt y') 
     = \sum_{z_1' \in O(z_1)}(z_1', \wt f_{\eta_1'}\otimes \wt y')
     = |O(z_1)|(z_1, \wt f_{\eta_1'}\otimes \wt y').  
\end{equation*}
Since $|O(z_1)|$ is divisible by $p$, (2.8.7) follows. 
It remains to check, for $\eta, \eta' \in \ul I$ with $i \in \eta$, that  

\begin{equation*}
\tag{3.8.1}
(\wt f_{\eta}, \wt f_{\eta'}) = \begin{cases}
                                  (1 - q_i^2)^{-|\eta|}   &\quad\text{ if } \eta = \eta', \\
                                  0                     &\quad\text{ if } \eta \ne \eta'.
                              \end{cases}
\end{equation*}

In fact, if $\eta \ne \eta'$, $\weit(\wt f_{\eta}) \ne \weit(\wt f_{\eta'})$. 
Hence $(\wt f_{\eta}, \wt f_{\eta'}) = 0$ by (1.3.3).   
Assume that $\eta = \eta'$. Write  $\eta = \zeta \sqcup \{ i\}$, and 
$\wt f_{\eta} = f_i\wt f_{\zeta}$, where $\wt f_{\zeta} = \prod_{j \in \zeta}f_j$.
Since $f_i$ are commuting each other for $i \in \eta$, 
we have ${}_ir(\wt f_{\eta}) = \wt f_{\zeta}$. Then 
$(\wt f_{\eta}, \wt f_{\eta}) = (f_i, f_i)(\wt f_{\zeta}, \wt f_{\zeta})$. 
By induction on $|\zeta_1|$ for any subset $\zeta_1 \subset \eta$, we have
\begin{equation*}
(\wt f_{\eta}, \wt f_{\eta}) = \prod_{i \in \eta}(f_i, f_i) = (1 - q_i^2)^{-|\eta|}.
\end{equation*}
Thus (3.8.1) holds. This proves (i). The proposition is proved. 
\end{proof}

\para{3.9.}
In order to prove the surjectivity of $\Phi$, we need 
a preliminary for Kashiwara operators.
Assume that $(\a_i,\a_j) = 0$.  Since ${}_ir$ and ${}_jr$ commute each other,
and the left action of $f_i$ commutes with ${}_jr$ by Lemma 1.4, 
$\Pi_{i,a}$ and $\Pi_{j,b}$ in  (1.7.1) commute
each other. Take $x \in \BU_q^-$, and write it as 
$x = \sum_{n \ge 0}f_i^{(n)}x_n$ with $x_n \in \BU_q^-[i]$. We also write 
$x_n = \sum_{m \ge 0}f_j^{(m)}x_{n, m}$ with $x_{n,m} \in \BU_q^-[j]$, hence 
\begin{equation*}
\tag{3.9.1}
x = \sum_{n,m \ge 0}f_i^{(n)}f_j^{(m)}x_{n,m}.
\end{equation*}  
Then $x_{n,m} = q_i^{n(n-1)/2}q_j^{m(m -1)/2}\Pi_{j,m}\Pi_{i,n}(x)$. 
Since $\Pi_{j,m}\Pi_{i,n}(x) = \Pi_{i,n}\Pi_{j,m}(x)$, we have 
$x_{n,m} = x_{m,n} \in \BU_q^-[i] \cap \BU_q^-[j]$.  
Since the expression of $x$ in (3.9.1) is unique, we have a direct sum 
decomposition 
\begin{equation*}
\tag{3.9.2}
\BU_q^- = \bigoplus_{n, m \ge 0}f_i^{(n)}f_j^{(m)}(\BU_q^-[i] \cap \BU_q^-[j]).
\end{equation*}
If $x \in {}_{\BA}\BU_q^-$, 
then $x_{n,m} \in {}_{\BA}\BU_q^-$. Hence (3.9.2) implies that 
\begin{equation*}
\tag{3.9.3}
{}_\BA\BU_q^- = \bigoplus_{n,m \ge 0}f_i^{(n)}f_j^{(m)}
                  {}_{\BA}(\BU_q^-[i] \cap \BU_q^-[j]) 
\end{equation*}

By  a similar argument as in the proof of (3.9.3), we have
the following lemma.
\begin{lem} 
For $\eta \in \ul I$, set $\BU_q^-[\eta] = \bigcap_{i \in \eta}\BU_q^-[i]$.
Then we have
\begin{equation*}
\tag{3.10.1}
{}_{\BA}\BU_q^- = \bigoplus_{(a_i) \in \BN^{\eta}}
         \biggl(\prod_{i \in \eta}f_i^{(a_i)}\biggr){}_{\BA}(\BU_q^-[\eta]).
\end{equation*}
\end{lem}

\para{3.11.}
Let $\bB$ be the canonical basis of $\BU_q^-$, and 
$E_i, F_i : \bB \to \bB \cup \{ 0\}$ be Kashiwara operators on $\bB$.  
We consider the partition $\bB = \bigsqcup_{n \ge 0}\bB_{i;n}$.
For $b \in \bB$, we have $\ve_{\s(i)}(\s(b)) = \ve_i(b)$ by definition.
Hence we have $\s(\bB_{i;n}) = \bB_{\s(i); n}$. 
This implies that 
\begin{equation*}
\tag{3.11.1}
\s \circ E_i \circ \s\iv  = E_{\s(i)}, \qquad 
\s \circ F_i \circ \s\iv = F_{\s(i)}.
\end{equation*}

\begin{lem}  
Assume that $(\a_i, \a_j) = 0$.  Then for $b \in \bB_{i;0} \cap \bB_{j;0}$,
we have $\pi_{i;n}\pi_{j;m}(b) = \pi_{j;m}\pi_{i;n}(b)$.  
The map $\pi_{i;n}\pi_{j;m}$ gives a bijection 
\begin{equation*}
\tag{3.12.1}
\pi_{i;n}\pi_{j;m} :\bB_{i;0} \cap \bB_{j;0} \to \bB_{i;n} \cap \bB_{j;m}.
\end{equation*}
\end{lem}

\begin{proof}
Take $b \in \bB_{i;0} \cap \bB_{j;0}$. 
Set $b_1 = \pi_{i;n}(b)$ and $b_2 = \pi_{j;m}b_1$.  
Then $b_1$ is the unique element in $\bB$ 
such that $b_1 \equiv f_i^{(n)}b \mod f_i^{n+1}\BU_q^-$. 
Since $\ve_j(b_1) = 0$ by (3.9.2),  
$b_2$ is the unique element in $\bB$ such that 
$b_2 \equiv f_j^{(m)}b_1 \mod f_j^{m +1}\BU_q^-$.
By (C5), $b_1 \equiv (b_1)_{[i;n]} = (f_i^{(n)}b)_{[i;n]} \mod q\SL_{\BZ}(\infty)$, 
and $b_2 \equiv (b_2)_{[j;m]} = (f_j^{(m)}b_1)_{[j;m]} \mod q\SL_{\BZ}(\infty)$. 
By (3.9.3), this implies that  
\begin{equation*}
\tag{3.12.2}
b_2 \equiv (f_j^{(m)}f_i^{(n)}b)_{[n,m]} \mod q\SL_{\BZ}(\infty),
\end{equation*}
where $x_{[n,m]}$ denotes the projection of $x \in {}_{\BA}\BU_q^-$ onto
$f_i^{(n)}f_j^{(m)}{}_{\BA}(\BU_q^-[i] \cap \BU_q^-[j])$. 
If we consider $b_2' = \pi_{i;n}\pi_{j;m}b$, we obtain a similar formula 
as (3.12.2) by replacing $b_2$ by $b_2'$.
Since $b_2$ and $b_2'$ are bar-invariant, this implies that 
$b_2 = b_2'$. Thus we have proved $\pi_{i,n}\pi_{j;m} = \pi_{j;m}\pi_{i;n}$. 
$\pi_{i;n}\pi_{j;m}$ gives a map $\bB_{i;0}\cap \bB_{j;0} \to \bB_{i;n} \cap \bB_{j;m}$. 
This map is bijective since $\pi_{j;m}$ gives a bijection 
$\bB_{i;0} \cap \bB_{j;0} \to \bB_{i;0} \cap \bB_{j;m}$, and 
$\pi_{i;n}$ gives a bijection $\bB_{i;0} \cap \bB_{j;m} \to \bB_{i;n} \cap \bB_{j;m}$. 
The lemma is proved. 
\end{proof}

\para{3.13.}
If $b \in \bB^{\s}$, $\ve_i(b)$ is constant for $i \in \eta$ by 3.11, 
which we denote by $\ve_{\eta}(b)$.  We define
\begin{equation*}
\tag{3.13.1}
\bB^{\s}_{\eta;a} = \{ b \in \bB^{\s} \mid \ve_{\eta}(b) = a\}, 
\qquad \bB^{\s}_{\eta; > a} = \bigsqcup_{a' > a}\bB^{\s}_{\eta;a'}.
\end{equation*} 
We have a partition $\bB^{\s} = \bigsqcup_{a \ge 0}\bB^{\s}_{\eta;a}$. 
By Lemma 3.12, one can define a bijection 
$\pi_{\eta;a} : \bB^{\s}_{\eta;0} \isom \bB^{\s}_{\eta;a}$, 
where $\pi_{\eta;a}$ is the restriction of $\prod_{i \in \eta}\pi_{i;a}$ on 
$\bB^{\s}_{\eta;0}$. 
We define Kashiwara operators $\wt F_{\eta}, \wt E_{\eta} :
 \bB^{\s} \to \bB^{\s} \cup \{ 0\}$ in a similar way as 1.14 by using $\pi_{\eta; a}$. 
Note that $\wt F_{\eta}$ coincides with the restriction of $\prod_{i \in \eta}F_i$
on $\bB^{\s}$, and similarly for $\wt E_{\eta}$.   
\par
Take $b \in \bB^{\s}_{\eta;0}$, and set $b' = \pi_{\eta;a}b$. By using the
decomposition (3.10.1) and the property (C7) of the canonical basis, we see that 
\begin{equation*}
\tag{3.13.2}
\wt f_{\eta}^{(a)}b = \bigl(\prod_{i \in \eta}f_i^{(a)}\bigr)b \equiv b' \mod Z_{\eta; > a}
\end{equation*}
with  $Z_{\eta; > a} = 
     \sum_{(a_i)}\bigl(\prod_{i \in \eta}f_i^{(a_i)}\bigr) {}_{\BA}\BU_q^-$,
where $(a_i)$ runs over all the elements in $\BN^{\eta}$ such that 
$a_i \ge a$ for all $i \in \eta$ and that $a_i >  a$ for some $i$.
Note that $Z_{\eta; > a}$ is $\s$-invariant. 
It follows from (1.9.1) and (C7), $\bigcup_{n' \ge n}\bB_{i;n'}$ gives 
an $\BA$-basis of 
$\sum_{n' \ge n}f_i^{(n')}{}_{\BA}\BU_q^-$.  This implies that 
\par\medskip\noindent
(3.13.3) \ Set $\bB_{\eta; > a} = \bigcup_{(a_i)}\{ b \in \bB \mid \ve_i(b) = a_i\}$, 
where $(a_i)$ runs over all the elements in $\BN^{\eta}$ such that $a_i \ge a$ for 
any $i \in \eta$, and that $a_i > a$ for some $i$.  
Then $\bB_{\eta; > a}$ gives an $\BA$-basis of $Z_{\eta; > a}$. 
\par\medskip

\begin{prop}   
The map $\Phi : {}_{\BA'}\ul\BU_q^- \to \BV_q$ is surjective.
\end{prop}

\begin{proof}
The proposition is proved in a similar way as the proof 
of Proposition 2.10 in [SZ2]. 
We know that the image of $\bB^{\s}$ gives a basis of $\BV_q$. 
Thus it is enough to see, for each $b \in \bB^{\s}$, that
\begin{equation*}
\tag{3.14.1}
\pi(b) \in \Im \Phi.
\end{equation*} 
Take $b \in \bB^{\s}$. Assume that $b \in \bB_{\nu}$.  
If $\nu = 0$, then $b = 1$, and (3.14.1) holds.  So assume that $\nu \ne 0$,
and by induction on $|\nu|$, we may assume that (3.14.1) holds 
for $b' \in \bB_{\nu'}$ such that $|\nu'| < |\nu|$. 
There exists $i \in I$ such that $a = \ve_i(b) > 0$ by (C6). 
We may also assume (by the backward induction on $a$) that, 
if $b' \in \bB^{\s} \cap \bB_{\nu}$ is such that 
$\ve_i(b') > a$, then $b'$ satisfies (3.14.1) 
(note that $\bB^{\s} \cap \bB_{\nu}$ is a finite set, 
hence the set of such $b'$ is empty if $a >> 0$). 
Let $\eta$ be the $\s$-orbit containing $i$.
Then $b \in \bB^{\s}_{\eta; a}$.  We consider $\pi_{\eta;a}\iv (b) = b' \in \bB^{\s}_{\eta;0}$. 
By applying (3.13.2), there exists $z \in Z^{\s}_{\eta; > a}$ such that 
\begin{equation*}
b = \wt f_{\eta}^{(a)}b' + z.
\end{equation*}
 By induction, $\pi(b') \in \Im \Phi$. 
Hence $\pi(\wt f_{\eta}^{(a)}b') = \Phi(\ul f_{\eta}^{(a)})\pi(b') \in \Im \Phi$. 
On the other hand, by (3.13.3), $z$ is written as an $\BA$-linear combination of 
the orbit sum  of $b'' \in \bB_{\eta, > a}$. 
If $b''$ is not $\s$-stable, its orbit sum  is contained in $J$, hence $\pi(b'') = 0$.
If $\s(b'') = b''$, then $b'' \in \bB^{\s} \cap \bB_{\nu}$ satisfies the condition  
$\ve_i(b'') > a$ for $i \in \eta$.  
Hence by induction hypothesis, $\pi(b'') \in \Im \Phi$. It follows that 
$\pi(z) \in \Im\Phi$, and we conclude that $\pi(b) \in \Im \Phi$. 
The proposition is proved.
\end{proof}

\para{3.15.}
Theorem 3.4 is now proved by Proposition 3.8 and 
Proposition 3.14 (modulo Proposition 3.5).  
Note that the discussion in the proof of Theorem 3.4 works 
also for the case where $p = 2$. 

\section{ Canonical bases }

\para{4.1.}
We keep the setup in 3.2.
Thus the order of $\s : I \to I$ 
is a power of $p$, and we consider $\BU_q^-$ and $\ul\BU_q^-$.   
We assume the existence 
of the canonical basis (resp. the canonical signed basis) if $p \ne 2$ 
(if $p = 2$). 
Let $\bB$ be the canonical basis of $\BU_q^-$.   
Our aim  is to construct the canonical basis $\ul\bB$
of $\ul\BU_q^-$ and a bijection $\bB^{\s} \simeq \ul\bB$, by making 
use of the isomorphism $\Phi: {}_{\BA'}\ul\BU_q^- \isom \BV_q$. 
Note that in considering ${}_{\BA'}\BU_q^-$ or $\BV_q$ with $p = 2$, 
there is no difference between the canonical basis and the canonical signed basis.
In the discussion below, we basically consider the case where $p \ne 2$, but 
all the discussion works also for $p = 2$, except Lemma 4.23.

\para{4.2.}
We consider the direct sum decomposition of 
${}_{\BA}\BU_q^-$ given in Lemma 3.10, and 
consider the action $\s$ on ${}_{\BA}\BU_q^-$. 
Then ${}_{\BA}(\BU_q^-[\eta])$ is $\s$-stable, and $\s$ permutes 
the factors $\prod_{i \in \eta}f_i^{(a_i)}$. Those factors are $\s$-invariant
if and only if $a_i$ has a constant value $a$ for any $i \in \eta$, in that case 
$\prod_if_i^{(a_i)} = \wt f_{\eta}^{(a)}$.  If we replace $\BA$ by $\BA'$, 
a similar formula as in (3.10.1) holds for $\BA'$.  Thus we have 

\begin{equation*}
\tag{4.2.1}
{}_{\BA'}\BU_q^{-,\s} \equiv \bigoplus_{a \in \BN}\wt f_{\eta}^{(a)}
            {}_{\BA'}(\BU_q^-[\eta]^{\s}) \quad \mod J. 
\end{equation*} 

Recall that $\pi : {}_{\BA'}\BU_q^{-,\s} \to \BV_q$ is the projection. 
We define a submodule $\BV_q[\eta]$ of $\BV_q$ by 
$\BV_q[\eta] = \pi(_{\BA'}\BU_q^-[\eta]^{\s})$.    
By using the $\BA'$-basis $\pi(\bB^{\s})$ of $\BV_q$,  
the following lemma is easily obtained from (4.2.1).

\begin{lem}  
For $\eta \in \ul I$, we have
\begin{equation*}
\tag{4.3.1}
\BV_q = \bigoplus_{a \in \BN}g_{\eta}^{(a)}\BV_q[\eta].
\end{equation*}
\end{lem}

\para{4.4.}
We consider the direct sum decomposition of ${}_{\BA}\ul\BU_q^-$ 
as in (1.9.1).  Since $\Phi\iv(\pi(\bB^{\s}))$ gives an $\BA'$-basis of 
${}_{\BA'}\ul\BU_q^-$, this induces a direct sum decomposition
\begin{equation*}
\tag{4.4.1}
{}_{\BA'}\ul\BU_q^- = \bigoplus_{a \ge 0}\ul f_{\eta}^{(a)}{}_{\BA'}\ul\BU_q^-[\eta].
\end{equation*}

Comparing (4.3.1) and (4.4.1), together with Theorem 3.4, we have the following. 

\begin{cor}  
Assume that $\eta \in \ul I$. 
The map $\Phi : {}_{\BA'}\ul\BU_q^- \to \BV_q$ induces an isomorphism of $\BA'$-modules, 
\begin{equation*}
\tag{4.5.1}
\Phi : \ul f_{\eta}^{(a)}{}_{\BA'}\ul\BU_q^-[\eta]  \to g_{\eta}^{(a)}\BV_q[\eta].
\end{equation*}
\end{cor} 

\para{4.6.}
Recall that $\SL_{\BZ}(\infty)$ is 
a $\BZ[q]$-submodule of ${}_{\BA}\BU_q^-$ spanned by $\bB$.  Then 
$\s$ acts on $\SL_{\BZ}(\infty)$, and $\pi(\SL_{\BZ}(\infty)^{\s})$ 
gives an $\BF[q]$-submodule of $\BV_q$ spanned by $\pi(\bB^{\s})$. 
We denote $\pi(\SL_{\BZ}(\infty)^{\s})$ by ${}_{\BF[q]}\BV_q$ and 
$\pi(\bB^{\s})$ by $\bB^{\dia}$.  
$\bB^{\dia}$ gives an $\BF[q]$-basis of ${}_{\BF[q]}\BV_q$. 
\par
Let $\eta \in \ul I$. 
By (3.10.1), we have a direct sum decomposition 
\begin{equation*}
\tag{4.6.1}
\wt f_{\eta}^{(a)}{}_{\BA}\BU_q^- = \wt f_{\eta}^{(a)}{}_{\BA}(\BU_q^-[\eta]) \oplus 
           \bigoplus_{(a_i)}(\prod_{i \in \eta}f_i^{(a_i)}){}_{\BA}(\BU_q^-[\eta]),     
\end{equation*}
where $(a_i)_{i \in \eta}$ is such that $a_i \ge a$ for any $i$, 
and that $a_i > a$ for some $i$.  
Note that the second term coincides with $Z_{\eta; > a}$ given in (3.13.2).
For $x \in \wt f_{\eta}^{(a)}{}_{\BA}\BU_q^-$, 
let $x_{[\eta; a]}$ 
be the projection of $x$ onto $\wt f_{\eta}^{(a)}{}_{\BA}(\BU_q^-[\eta])$.  
By the discussion in 3.9, if 
$x \in \SL_{\BZ}(\infty)$, then $x_{[\eta; a]}$ belongs to $\SL_{\BZ}(\infty)$. 
$\wt f_{\eta}^{(a)}\BU_q^-$ has a basis 
$\bB_{\eta;\ge a} = \bB_{\eta; > a} \sqcup \bB_{\eta;a}$, where 
$\bB_{\eta;a} = \bigcap_{i \in \eta}\bB_{i; a}$.
Let ${}_{\BZ[q]}\wt f_{\eta}^{(a)}\BU_q^-$ be the $\BZ[q]$-submodule
of $\wt f_{\eta}^{(a)}\BU_q^-$ spanned by $\bB_{\eta; \ge a}$. Then 
${}_{\BZ[q]}\wt f_{\eta}^{(a)}\BU_q^-$ is a $\BZ[q]$-submodule of $\SL_{\BZ}(\infty)$. 
On the other hand, $\{ b_{[\eta;a]} \mid b \in \bB_{\eta;a}\}$ gives an $\BA$- basis 
of ${}_{\BA}\wt f_{\eta}^{(a)}\BU_q^-[\eta]$, and $\bB_{\eta; > a}$ gives an $\BA$-basis of
$Z_{\eta; > a}$. 
We denote by ${}_{\BZ[q]}\wt f_{\eta}^{(a)}\BU_q^-[\eta]$ the $\BZ[q]$-submodule of 
$\wt f_{\eta}^{(a)}\BU_q^-[\eta]$ spanned by $\{ b_{[\eta,a]} \mid b \in \bB_{\eta;a}\}$, 
and by ${}_{\BZ[q]}Z_{\eta; >a}$
the $\BZ[q]$-submodule of $Z_{\eta; > a}$ spanned by $\bB_{\eta; > a}$.  Thus we have
a decomposition as $\BZ[q]$-submodules of $\SL_{\BZ}(\infty)$, 
\begin{equation*}
\tag{4.6.2}
{}_{\BZ[q]}\wt f_{\eta}^{(a)}\BU_q^- = 
         {}_{\BZ[q]}\wt f_{\eta}^{(a)}\BU_q^-[\eta] \oplus {}_{\BZ[q]}Z_{\eta; >a}.
\end{equation*}   

Note that ${}_{\BZ[q]}\wt f_{\eta}^{(a)}\BU_q^-[\eta]$ 
coincides with $\wt f_{\eta}^{(a)}({}_{\BZ[q]}\BU_q^-[\eta])$. 
Also note that (4.6.2) holds if we replace $\eta$ by any subset $\g \subset \eta$,
and define $\wt f_{\g}^{(a)}\BU_q^-[\g], Z_{\g; >a}$ accordingly. 
\par 
We show a lemma.

\begin{lem}  
Assume that $b \in \bB_{\eta; a}$.  Then 
$b - b_{[\eta;a]} \in q\SL_{\BZ}(\infty)$.
\end{lem} 

\begin{proof}
We consider the following statement.  
\par\medskip\noindent
(4.7.1) \ For any subset $\g \subset \eta$, 
$b - b_{[\g;a]} \in q\SL_{\BZ}(\infty)$. 
\par\medskip
We prove (4.7.1) by induction on $|\g|$. If $|\g| = 1$, it is certainly true by 
(1.12.1).  So assume that (4.7.1) holds for $\g \subset \eta$, and prove 
that it holds for $\g' = \g \sqcup \{ j\} \subset \eta$. 
By using the induction hypothesis, and (4.6.2) for $\g$, one can write as
$b = b_{[\g;a]} + \sum_{b'}a_{b'}b'$, where $b' \in \bB_{\g; > a}$ and 
$a_{b'} \in q\BZ[q]$.  
We consider the decomposition 
$f_j^{(a)}\BU_q^- = f_j^{(a)}\BU_q^-[j] \oplus f_j^{(a+1)}\BU_q^-$. 
Write $b_{[\g;a]} = x_1 + x_2, b' = y'_1 + y'_2$ according to this decomposition, 
where $x_1, y_1' \in f_j^{(a)}\BU_q^-[j]$ and 
$x_2, y_2' \in f_j^{(a+1)}\BU_q^-$.  Then $x_1$ coincides with $b_{[\g';a]}$, and 
$x_1 + \sum_{b'}a_{b'}y_1'$ coincides with $b_{[j;a]}$. 
By (1.12.1), $b \equiv b_{[j;a]} \mod q\SL_{\BZ}(\infty)$.  
Since $y_1' \in \SL_{\BZ}(\infty)$, $\sum_{b'}a_{b'}y_1' \in q\SL_{\BZ}(\infty)$. 
It follows that $b \equiv x_1 = b_{[\g';a]} \mod q\SL_{\BZ}(\infty)$, and 
(4.7.1) holds for $\g'$.  The lemma is proved. 
\end{proof}

\para{4.8.}
Let ${}_{\BF[q]}\BV_q$ be the $\BF[q]$-submodule of $\BV_q$ spanned by $\bB^{\dia}$ 
as in 4.6.
Take $\eta \in \ul I$. We consider the decomposition of $\BV_q$ 
as in Lemma 4.3, and  for $x \in \BV_q$, 
we denote by $x_{[\eta;a]}$ the projection of $x$ onto $g_{\eta}^{(a)}\BV_q[\eta]$. 
Then $\{ b_{[\eta:a]} \mid b \in \bB^{\dia}_{\eta;a}\}$ gives an $\BA'$-basis of 
$g^{(a)}_{\eta}\BV_q[\eta]$, and $\bB^{\dia}_{\eta; > a}$ gives an $\BA'$-basis 
of $\sum_{a' >  a}g^{(a')}_{\eta}\BV_q$ (here set $\bB_{\eta;a}^{\dia} = \pi(\bB_{\eta; a}^{\s})$
and $\bB^{\dia}_{\eta; > a} = \pi(\bB_{\eta; > a}^{\s})$).
For each $b \in \bB^{\dia}_{\eta; a}$, $b_{[\eta;a]}$ belongs to 
${}_{\BF[q]}\BV_q$.   We denote by ${}_{\BF[q]}g_{\eta}^{(a)}\BV_q[\eta]$ 
the $\BF[q]$-submodule of ${}_{\BF[q]}\BV_q$ spanned by $b_{[\eta;a]}$, which 
coincides with $g_{\eta}^{(a)}{}_{\BF[q]}\BV_q[\eta]$.    
Thus we have a decomposition as $\BF[q]$-submodule of ${}_{\BF[q]}\BV_q$, 
\begin{equation*}
\tag{4.8.1}
\sum_{a' \ge a}g_{\eta}^{(a')}{}_{\BF[q]}\BV_q
       = g_{\eta}^{(a)}{}_{\BF[q]}\BV_q[\eta] \oplus 
           \sum_{a' > a}g_{\eta}^{(a')}{}_{\BF[q]}\BV_q
\end{equation*}

As a corollary to Lemma 4.7, we have the following.

\begin{cor}  
For $b \in \bB^{\dia}_{\eta; a}$, we have 
$b - b_{[\eta;a]} \in q({}_{\BF[q]}\BV_q)$. 
\end{cor}

\para{4.10.}
We consider the $\BA'$-algebra isomorphism 
$\Phi:{}_{\BA'}\ul\BU_q^- \isom \BV_q$, and define 
$\BA'$-basis $\ul\bB^{\bullet}$ of ${}_{\BA'}\ul\BU_q^-$ by
$\ul\bB^{\bullet} = \Phi\iv(\bB^{\dia})$. 
Consider the decomposition ${}_{\BA'}\ul\BU_q^- = 
    \bigoplus_{a \ge 0}\ul f_{\eta}^{(a)}{}_{\BA'}\ul\BU_q^-[\eta]$ in (4.4.1).
Then by Corollary 4.5, $\{ b_{[\eta;a]} \mid b \in \ul\bB^{\bullet}_{\eta;a}\}$ gives 
an $\BA'$-basis of $\ul f_{\eta}^{(a)}\ul\BU_q^-[\eta]$,  
and $\ul\bB^{\bullet}_{\eta; >a}$ gives an $\BA'$-basis of 
$\sum_{a' > a}\ul f^{(a')}_{\eta}{}_{\BA'}\ul\BU_q^-$. 
Here we set $\ul\bB^{\bullet}_{\eta;a} = \Phi\iv(\bB^{\dia}_{\eta;a})$, and 
$\ul\bB^{\bullet}_{\eta; > a} = \Phi\iv(\bB^{\dia}_{\eta; > a})$.  
Let $\ul\SL_{\BF}(\infty)$ be the  $\BF[q]$-submodule of ${}_{\BA'}\ul\BU_q^-$
spanned by $\ul\bB^{\bullet}$. Then for each $b \in \ul\bB^{\bullet}_{\eta;a}$, 
$b_{[\eta;a]}$ belongs to $\ul\SL_{\BF}(\infty)$, and we denote by 
${}_{\BF[q]}\ul f^{(a)}_{\eta}\ul\BU_q^-[\eta]$ the $\BF[q]$-submodule
of ${}_{\BA'}\ul f_{\eta}^{(a)}\ul\BU_q^-[\eta]$ spanned by $b_{[\eta;a]}$ 
for $b \in \bB^{\bullet}_{\eta;a}$, which coincides with 
$\ul f_{\eta}^{(a)}{}_{\BF[q]}\ul\BU_q^-[\eta]$. 
Thus we have a decomposition as $\BF[q]$-submodules of $\ul\SL_{\BF}(\infty)$,

\begin{equation*}
\tag{4.10.1}
\sum_{a' \ge a}\ul f_{\eta}^{(a')}{}_{\BF[q]}\ul\BU_q^-
     = \ul f^{(a)}_{\eta}{}_{\BF[q]}\ul\BU_q^-[\eta] 
         \oplus \sum_{a' >  a}\ul f_{\eta}^{(a')}{}_{\BF[q]}\ul\BU_q^-.
\end{equation*}

Corollary 4.9 can be rewritten as follows;

\begin{cor}  
For $b \in \ul\bB^{\bullet}_{\eta;a}$, we have
$b - b_{[\eta;a]} \in q\ul\SL_{\BF}(\infty)$. 
\end{cor}

\para{4.12.}
Since $\bB$ is almost orthonormal, 
$\bB^{\dia}$ gives an almost orthonormal basis 
of $\BV_q$ in the sense that, for $b, b' \in \bB^{\dia}$, 
\begin{equation*}
\tag{4.12.1}
(b, b') \in \d_{b,b'} + (q\BF[[q]] \cap \BF(q)). 
\end{equation*} 
$\Phi$ is compatible with the bilinear forms on ${}_{\BA'}\ul\BU_q^-$
and on $\BV_q$ by Proposition 3.8 (i).  
Hence $\ul\bB^{\bullet}$ gives an almost orthonormal basis 
of ${}_{\BA'}\ul\BU_q^-$ in the sense of (4.12.1). 
\par
Moreover, since $\pi$ and $\Phi$ are compatible with the bar-involution, 
any $b \in \ul\bB^{\bullet}$ is bar-invariant. 
\par 
Recall that $\ul\bB^{\bullet}_{\eta;a} = \Phi\iv(\bB^{\dia}_{\eta;a})$ 
for each $\eta \in \ul I, a \ge 0$. 
Kashiwara operators $\wt E_{\eta}, \wt F_{\eta}$ on $\bB^{\dia}$ induces 
Kashiwara operators $\ul E_{\eta}, \ul F_{\eta}$ on $\ul\bB^{\bullet}$.
$\ul F_{\eta}$ gives a bijection 
$\ul\bB^{\bullet}_{\eta; a} \isom  \ul\bB^{\bullet}_{\eta; a+1}$, and 
$\ul E_{\eta}$ gives the inverse map  
$\ul\bB^{\bullet}_{\eta; a+1} \isom  \ul\bB^{\bullet}_{\eta; a}$.  
Then (3.13.2) implies the following;  
for each $b \in \ul\bB^{\bullet}_{\eta;0}$, 
set $b' = \ul F_{\eta}^a b \in \ul\bB^{\bullet}_{\eta;a}$.  Then we have
\begin{equation*}
\tag{4.12.2}
\ul f^{(a)}_{\eta}b \equiv b' \mod \ul Z^{\bullet}_{\eta, > a},
\end{equation*}
where $\ul Z^{\bullet}_{\eta;> a} = \sum_{a' > a}\ul f_{\eta}^{(a')}{}_{\BA'}\ul\BU_q^-$. 
Moreover, the map $b \mapsto b'$ gives a bijection 
\begin{equation*}
\tag{4.12.3}
\pi_{\eta; a} : \ul\bB^{\bullet}_{\eta;0} \isom \ul\bB^{\bullet}_{\eta;a}.
\end{equation*}

\par
The $*$-operation on $\BU_q^-$ induces the $*$-operation on 
$\BV_q$, which leaves $\wt f_{\eta}^{(a)}$ invariant. Hence $\Phi$ 
is compatible with $*$-operations on ${}_{\BA}\ul\BU_q^-$ and on $\BV_q$. 
Since $*(\bB) = \bB$, we have $*(\bB^{\dia}) = \bB^{\dia}$, and so
\begin{equation*}
\tag{4.12.4}
*(\ul\bB^{\bullet}) = \ul\bB^{\bullet}.
\end{equation*} 

\para{4.13.}
Let $Q = \bigoplus_{i \in I}\BZ \a_i$ be the root lattice of $X$, and 
$\ul Q = \bigoplus_{\eta \in \ul I}\BZ \a_{\eta}$ the root lattice of $\ul X$.
$\s$ acts on $Q$ by $\a_i \mapsto \a_{\s(i)}$, and we have $Q^{\s} \simeq \ul Q$ 
under the map $\sum_{i \in \eta}\a_{i} \mapsto \a_{\eta}$.  
If $b \in \bB^{\s}$, then the weight of $b$ is contained in $Q^{\s}$, and 
we have a partition $\bB^{\s} = \bigsqcup_{\nu \in Q_-^{\s}}\bB^{\s}_{\nu}$, 
where $\bB^{\s}_{\nu} = \bB^{\s} \cap \bB_{\nu}$. 
This gives a partition $\bB^{\dia} = \bigsqcup_{\nu \in Q_-^{\s}}\bB^{\dia}_{\nu}$, 
where $\bB^{\dia}_{\nu} = \pi(\bB^{\s}_{\nu})$. It follows that 
$\BV_q$ has a weight space decomposition $\BV_q = \bigoplus_{\nu \in Q_-^{\s}}(\BV_q)_{\nu}$, 
where $(\BV_q)_{\nu}$ is an $\BA'$-subspace of $\BV_q$ spanned by $\bB^{\dia}_{\nu}$.
On the other hand, the weight space decomposition 
$\ul\BU_q^- = \bigoplus_{\nu \in \ul Q_-}(\ul\BU_q^-)_{\nu}$ induces a weight space 
decomposition ${}_{\BA'}\ul\BU_q^- = \bigoplus_{\nu \in \ul Q_-}({}_{\BA'}\ul\BU_q^-)_{\nu}$.
The map $\Phi : {}_{\BA'}\ul\BU_q^- \isom \BV_q$ is compatible with those 
weight space decompositions under the identification $\ul Q \simeq Q^{\s}$. 
We have a partition $\ul\bB^{\bullet} = \bigsqcup_{\nu \in \ul Q_-}\ul\bB^{\bullet}_{\nu}$, 
where $\ul\bB^{\bullet}_{\nu} = \Phi\iv(\bB^{\dia}_{\nu}) 
= \ul\bB^{\bullet} \cap ({}_{\BA'}\ul\BU_q^-)_{\nu}$.  

\para{4.14.}
Let $\BZ_p$ be the ring of $p$-adic integers, and set $\BA_p = \BZ_p[q, q\iv]$, 
which is a ring containing $\BA = \BZ[q,q\iv]$. We have a natural surjective 
map $\BA_p \to \BA_p/p\BA_p \simeq \BA' = \BF[q,q\iv]$. 
Let us consider ${}_{\BA_p}\ul\BU_q^- = \BA_p\otimes_{\BA}\ul\BU_q^-$, which is 
an extension of ${}_{\BA}\ul\BU_q^-$.
Let $\vf : {}_{\BA}\ul\BU_q^- \to {}_{\BA}\ul\BU_q^-/p({}_{\BA}\ul\BU_q^-) \simeq 
{}_{\BA'}\ul\BU_q^-$ be the natural surjective map.  
$\vf$ is extended to the surjective map 
${}_{\BA_p}\ul\BU_q^- \to {}_{\BA'}\ul\BU_q^-$, 
which we also denote by $\vf$. 
$\vf$ is compatible with the bar-involutions. 
\par
Let $M$ be an $\BA$-submodule of ${}_{\BA}\ul\BU_q^-$ such that $\ol M = M$, 
and assume that 
$\vf(M) = {}_{\BA'}M$ is spanned by a finite subset $\ul\bB^{\bullet}_M$ 
of $\ul\bB^{\bullet}$.  Set ${}_{\BA_p}M = \BA_p\otimes_{\BA}M$. 
We have a surjective map $\vf : {}_{\BA_p}M \to {}_{\BA'}M$.   
For each $b \in \ul\bB^{\bullet}_M$, choose $x \in M$
such that $\vf(x) = b$.  Since $b$ is bar-invariant, 
$x - \ol x \in pM$. Thus one can write $x - \ol x = py$ 
for some $y \in M$.  $y$ satisfies the condition that 
$\ol y = -y$.  
Since $\vf(y) \in {}_{\BA'}M$ is written as an $\BA'$-linear combination 
of the basis $\ul\bB^{\bullet}_M$, there exists $u \in {}_{\BA'}M$
such that $\vf(y) = u - \ol u$. Hence $y$ is written as $y = y_1 - \ol y_1 + pz$
for some $y_1, z \in M$ such that $\ol z = -z$.  
Repeating this procedure, and taking the limit, one can find $c \in {}_{\BA_p}M$
such that $\ol c = c$ and that $\vf(c) = b$. For each $b \in \ul\bB^{\bullet}_M$, 
we choose $c$ as above, and let $\CB_M$ be the set of such  
$c\in {}_{\BA_p}M$  obtained from $\ul\bB^{\bullet}_M$. We show that 

\begin{lem}  
Let the notations be as above.
\begin{enumerate}
\item 
$\CB_M$ gives an $\BA_p$-basis of ${}_{\BA_p}M$. 
\item
Assume that $\CB_M \subset M$, and that they are almost orthonormal 
in the sense of 1.13. 
Then  $\CB_M$ gives an  $\BA$-basis of $M$.
\end{enumerate}
\end{lem}

\begin{proof}  
Let $M_0$ be the $\BA_p$-submodule of ${}_{\BA_p}M $ spanned by $\CB_M$.
Then we have ${}_{\BA_p}M  = M _0 + p({}_{\BA_p}M)$.  Hence by Nakayama's lemma, 
$M _0 = {}_{\BA_p}M$.   
We consider the relation that $\sum_{c \in \CB_M} a_cc = 0$ with $a_c \in \BA_p$. 
If this relation is non-trivial, we may assume that 
some of $a_c$ is not contained in $p\BA_p$.  But then $\sum_c \ul a_c \vf(c) = 0$, 
where $\ul a_c \in \BA'$ is the image of $a_c$, and so $\ul a_c = 0$ for all $c$.  
This shows that all $a_c \in p\BA_p$, a contradiction.
Thus $\CB_M$ is linearly independent.  (i) holds.
\par
Next assume that $\CB_M \subset M$, and they are almost orthonormal.
Take $x \in M$ and write as $x = \sum_{c \in \CB_M}a_cc$, with $a_c \in \BA_p$.
It is enough to show that $a_c \in \BA$ for any $c$.
Let $t$ be the smallest integer $\ge 0$ such that $q^ta_c \in \BZ_p[q]$ for any $c$.  
Assume that $c_1$ is such that $q^ta_{c_1} \in \BZ_p[q] - q\BZ_p[q]$.  
Then $q^t(x, c_1) \in q^ta_{c_1} + q\BZ_p[[q]]$.  On the other hand, 
$q^t(x, c_1) \in \BZ((q)) \cap \BQ(q)$.  It follows that the lowest term 
of $a_{c_1}$ has the form $\la q^{-t}$ with $\la \in \BZ$. 
Then we may replace $x$ by $x' = x - \la q^{-t}c_1 \in M$.    
Repeating this procedure, one can prove that $a_c \in \BA$ for any 
$c \in \CB_M$.  Hence (ii) holds.   The lemma is proved.   
\end{proof} 

\para{4.16.}
Let $b \in \ul\bB^{\bullet}_{\eta;0}$, and choose $c \in {}_{\BA_p}\ul\BU_q^-$
such that $\ol c = c$ and that $\vf(c) = b$.  Set 
$b' = \pi_{\eta:a} b \in \ul\bB^{\bullet}_{\eta; a}$.  We shall construct
$c' \in {}_{\BA_p}\ul\BU_q^-$ satisfying a similar property as in (4.12.2).  
Take $x \in {}_{\BA}\ul\BU_q^-$  such that $\vf(x) = b'$.  
Then $\vf(x - \ul f_{\eta}^{(a)}c) = 
  b' - \ul f_{\eta}^{(a)}b \in \ul Z^{\bullet}_{\eta, > a}$ by (4.12.2). 
This element is bar-invariant.
Hence by a similar argument as in 4.14 (applied for 
$M = \ul Z_{\eta; > a} = \sum_{a' > a}\ul f^{(a')}_{\eta}{}_{\BA}\ul\BU_q^-$), 
there exists $z \in {}_{\BA_p}{\ul Z}_{\eta;> a} 
       = \sum_{a' > a}\ul f_{\eta}^{(a')}{}_{\BA_p}\ul\BU_q^-$ with $\ol z = z$ 
such that $\vf(z) = b' - \ul f_{\eta}^{(a)}b$. It follows that  
$x - \ul f_{\eta}^{(a)}c  - z \in p({}_{\BA_p}\ul\BU_q^-)$.
Then there exists $y \in {}_{\BA}\ul\BU_q^-$ such that 
$\ol x  - x = py$.  Here $y$ satisfies the condition that $\ol y = -y$.  
Thus again, by a similar argument as in 4.14, there exists $c' \in {}_{\BA_p}\ul\BU_q^-$
such that $\ol c' = c'$ and that $c' = \ul f_{\eta}^{(a)}c + z$, 
namely $\vf(c') = b'$. 
\par
We consider the decomposition 
${}_{\BA_p}\ul\BU_q^- = \bigoplus_{a \ge 0}\ul f_{\eta}^{(a)}{}_{\BA_p}\ul\BU_q^-[\eta]$.
For $x \in {}_{\BA_p}\ul\BU_q^-$, we denote by $x_{[\eta;a]}$ the projection of $x$ onto
$\ul f_{\eta}^{(a)}{}_{\BA_p}\ul\BU_q^-[\eta]$. The following lemma holds. 

\begin{lem}  
Assume that $c \in {}_{\BA_p}\ul\BU_q^-$ is such that $\ol c = c$ and 
that $\vf(c) = b \in \ul\bB^{\bullet}_{\eta;0}$.
\begin{enumerate}  
\item
There exists $c' \in {}_{\BA_p}\ul\BU_q^-$  with $\ol c' = c'$ 
such that $\vf(c') = \pi_{\eta;a}b$ and that 
\begin{equation*}
\tag{4.17.1}
c' \equiv \ul f_{\eta}^{(a)}c \mod {}_{\BA_p}{\ul Z}_{\eta; > a}. 
\end{equation*} 
\item 
Under the notation in 4.16, we have
\begin{equation*}
\tag{4.17.2}
\ol{ (\ul f_{\eta}^{(a)}c)_{[\eta;a]}} - (\ul f_{\eta}^{(a)}c)_{[\eta;a]} 
     \in {}_{\BA_p}\ul Z_{\eta; > a}. 
\end{equation*}
\end{enumerate}
\end{lem}

\begin{proof}
(i) follows from the discussion in 4.16.  We show (ii). 
By (4.17.1), $(\ul f_{\eta}^{(a)}c)_{[\eta;a]} \equiv c' \mod {}_{\BA_p}\ul Z_{\eta; > a}$.
Since $\ol c' = c'$ and ${}_{\BA_p}\ul Z_{\eta; > a}$ is bar-invariant, 
we obtain (4.17.2).
\end{proof}

We shall prove the following theorem. 

\begin{thm} 
Assume that $p \ne 2$. Under the setup in 4.1,  the following holds.
\begin{enumerate}
\item \ 
$\ul\BU_q^-$ has the canonical basis $\ul\bB$.
\item \ The map $\vf : {}_{\BA}\ul\BU_q^- \to {}_{\BA'}\ul\BU_q^-$ 
gives a bijection $\vf : \ul\bB \isom \ul\bB^{\bullet}$. Hence 
the natural map ${}_{\BA}\BU_q^{-\s} \to \BV_q \to {}_{\BA'}\ul\BU_q^-$ induces 
a unique bijection $\xi : \bB^{\s} \isom \ul\bB$ compatible with Kashiwara operators, 
where 
\begin{equation*}
\tag{4.18.1}
\xymatrix{
  \xi : \bB^{\s} \ar[rr]^{\pi} &   &  
            \bB^{\dia} \ar[rr]^{\Phi\iv} &  &  
                \ul\bB^{\bullet} \ar[rr]^{\vf\iv} &  &  \ul\bB. 
}
\end{equation*}
\item \ $*(\ul\bB) = \ul\bB$. 
\end{enumerate} 
\end{thm}

\remark{4.19.}
In the case where $p = 2$, a weaker result than Theorem 4.18 holds.
Assume that there exists a basis $\CB$ of $\BU_q^-$ such that 
$\wt\CB = \CB \sqcup -\CB$ is the canonical signed basis, satisfying the 
properties 
(C1) $\sim$ (C6) in 1.12, and (C7$'$) below instead of (C7). 
\begin{description}
\item [(C7$'$)] \ Assume that $b \in \CB_{i;0}$. Then for $a > 0$, 
there exists a unique element $b' \in \CB_{i;a}$ such that 
\begin{equation*}
\pm b' \equiv f_i^{(a)}b \mod f_i^{a+1}\BU_q^-.
\end{equation*} 
The correspondence $b \mapsto b'$ gives a bijection 
$\pi_{i;a} : \CB_{i;0} \isom \CB_{i;a}$. 
\end{description}
\par\noindent
Then there exists a basis $\ul\CB$ of $\ul\BU_q^-$ 
such that $\wt{\ul\CB} = \ul\CB \sqcup -\ul\CB$
is the canonical signed basis, satisfying the properties (C1) $\sim$ (C6) and (C7$'$). 
Moreover, a similar result as in (ii) in the theorem holds, up to sign.  In particular 
there exists a unique bijection  $\xi : \wt\CB^{\s} \isom \wt{\ul\CB}$ compatible
with Kashiwara operators.  

\para{4.20.}
In the discussion below, we basically follow the setup in 4.1.  However, 
in the case where $p = 2$, we consider $\CB$ and (C7$'$) instead of $\bB$ and 
(C7).
\par 
Note that the properties corresponding to (C1) $\sim$ (C7) in 1.12
hold for $\ul\bB^{\bullet}$ in ${}_{\BA'}\ul\BU_q^-$,
by replacing $\BZ[q]$-module $\ul\SL_{\BZ}(\infty)$
by $\BF[q]$-module $\ul\SL_{\BF}(\infty)$.  In fact, (C1) and (C4) follows 
from the discussion in 4.10 and 4.13, (C2), (C3) and (C7) follows from that of 4.12.
(C5) follows from Corollary 4.11, and (C6) follows from 
Proposition 1.26, applied for $R = \BF(q)$.
\par  
We shall construct $\ul\bB_{\nu}$ by induction on $|\nu|$. 
If $\nu = 0$, $\ul\bB_{\nu} = \{ 1 \}$ satisfies all the 
properties. Thus we assume that $\nu \ne 0$, and that 
$\ul\bB_{\nu'}$ are constructed for all $\nu'$ such that $|\nu'| < |\nu|$.  
Note that we have a partition 
$\ul\bB^{\bullet} = \bigsqcup_{\nu \in \ul Q_-}\ul\bB^{\bullet}_{\nu}$.
We shall construct $\ul\bB_{\nu}$ such that $\vf$ gives a bijection 
$\ul\bB_{\nu} \to \ul\bB^{\bullet}_{\nu}$. 
\par
Take $b_{\bullet} \in \ul\bB^{\bullet}_{\nu}$.  By (C6) for $\ul\bB^{\bullet}$, 
there exists $\eta \in \ul I$ such that 
$\ve_{\eta}(b_{\bullet}) = a \ne 0$. 
Then by (4.12.2) and (4.12.3), there exists 
$b'_{\bullet} \in \ul\bB^{\bullet}_{\eta:0}$ such that 
$b_{\bullet} \equiv \ul f_{\eta}^{(a)}b'_{\bullet} 
        \mod \sum_{a' > a}\ul f_{\eta}^{(a')}{}_{\BA'}\ul\BU_q^-$. 
Since $b'_{\bullet} \in \ul\bB^{\bullet}_{\nu'}$ with $|\nu'| < |\nu|$, 
by induction hypothesis, there exists $b' \in \ul\bB_{\eta; 0}$ such that 
$\vf(b') = b'_{\bullet}$.  
Since $b' \in \ul\SL_{\BZ}(\infty)$, $b'_{[\eta;0]} \in \ul\SL_{\BZ}(\infty)$ 
by Lemma 1.11 (i).  Then 
$(\ul f_{\eta}^{(a)}b')_{[\eta;a]} = f_{\eta}^{(a)}(b'_{[\eta;0]}) \in \ul\SL_{\BZ}(\infty)$
by (1.9.2). 
Note that $(\ul\BU_q^-)_{\nu}$ is finite dimensional, 
$\ul f^{(n)}_{\eta}\ul\BU_q^-[\eta] \cap (\ul\BU_q^-)_{\nu} = \{ 0\}$ if 
$n >> 0$.  
We shall construct $\ul\bB_{\eta;a} \cap \ul\bB_{\nu}$ by backward induction on 
$a$. So assume that $\ul\bB_{\eta;a'}$ were already constructed for $a' > a$ 
and that $\bigcup_{a' > a}(\ul\bB_{\eta;a'} \cap \ul\bB_{\nu})$ 
gives an $\BA$-basis of 
$\sum_{a' > a}\ul f^{(a')}_{\eta}{}_{\BA}\ul\BU_q^- \cap (\ul\BU_q^-)_{\nu}$.
By the above discussion, $(\ul f_{\eta}^{(a)}b')_{[\eta;a]} \in \ul\SL_{\BZ}(\infty)$. 
By applying Lemma 4.17 (ii), we have
\begin{equation*}
\tag{4.20.1}
\ol{(\ul f_{\eta}^{(a)}b')_{[\eta;a]}} - (\ul f_{\eta}^{(a)}b')_{[\eta;a]}
          \in {}_{\BA_p}\ul Z_{\eta; > a} \cap {}_{\BA}\ul\BU_q^-.  
\end{equation*}
Set $x = \ol{(\ul f_{\eta}^{(a)}b')_{[\eta;a]}} - (\ul f_{\eta}^{(a)}b')_{[\eta;a]}$.
Then $x \in \sum_{a' > a}\ul f^{(a')}_{\eta}{}_{\BA}\ul\BU_q^- \cap (\ul\BU_q^-)_{\nu}$.
Hence we can write as $x = \sum_{b''}a_{b''}b''$, where 
$b'' \in \bigcup_{a' > a}(\ul\bB_{\eta; a'} \cap \ul\bB_{\nu})$, and 
$a_{b''} \in \BA$.
Since $\ol x = -x$, and all the $b''$ are bar-invariant, 
$a_{b''}$ is written as $a_{b''} = c_{b''} - \ol{c_{b''}}$ for some 
$c_{b''} \in q\BZ[q]$.  Set
$b = (\ul f_{\eta}^{(a)}b')_{[\eta;a]} + \sum_{b''}c_{b''}b''$. 
Then $b \in \ul\SL_{\BZ}(\infty)$ and $\ol b = b$.   
Since $b_{[\eta;a]} = (\ul f_{\eta}^{(a)}b')_{[\eta;a]}$, 
we see that 
\begin{equation*}
\tag{4.20.2}
b \equiv b_{[\eta;a]} \mod q\ul\SL_{\BZ}(\infty).
\end{equation*} 

We also note that

\begin{equation*}
\tag{4.20.3}
b \equiv \ul f_{\eta}^{(a)}b' \mod \ul f_{\eta}^{a+1}\ul\BU_q^-.
\end{equation*}

In fact, $b \equiv b_{[\eta;a]} \mod \ul f_{\eta}^{a+1}\ul\BU_q^-$, 
$\ul f_{\eta}^{(a)}b' \equiv (\ul f_{\eta}^{(a)}b')_{[\eta;a]}
   \mod \ul f_{\eta}^{a+1}\ul\BU_q^-$.
Since $b_{[\eta;a]} = (\ul f_{\eta}^{(a)}b')_{[\eta;a]}$, (4.20.3) follows.  

Next we show that
\begin{equation*}
\tag{4.20.4}
\vf(b) = b_{\bullet}.
\end{equation*}

In fact, note that 
\begin{align*}
\vf((\ul f_{\eta}^{(a)}b')_{[\eta;a]}) = (\vf(\ul f_{\eta}^{(a)}b'))_{[\eta;a]}
           = (\ul f_{\eta}^{(a)}b'_{\bullet})_{[\eta;a]} = (b_{\bullet})_{[\eta;a]}. 
\end{align*}
It follows that 
$\vf(b) = (b_{\bullet})_{[\eta;a]} + \sum_{b''}\ul c_{b''}b''_{\bullet}$, 
where $b''_{\bullet} = \vf(b'')$ and $\ul c_{b''} \in q\BF[q]$. 
Since $b_{\bullet} \equiv (b_{\bullet})_{[\eta;a]} \mod q\ul\SL_{\BF}(\infty)$
by Corollary 4.11,  
and $\vf(b), b_{\bullet}$ are bar-invariant, we obtain $\vf(b) = b_{\bullet}$. 
Hence (4.20.4) holds. 

\par
Set $\ul Z_{\eta; \ge a} = \sum_{a' \ge a}\ul f^{(a')}_{\eta}{}_{\BA}\ul\BU_q^-$. 
In the above discussion, for each 
$b_{\bullet} \in \ul\bB^{\bullet}_{\eta; a} \cap \ul\bB^{\bullet}_{\nu}$, 
we have constructed 
$b \in \ul\SL_{\BZ}(\infty) \cap \ul Z_{\eta; \ge a}$
such that $\vf(b) = b_{\bullet}$. 
We define $\ul\bB_{\eta;a} \cap \ul\bB_{\nu}$ as the set of those $b$ corresponding to 
$b_{\bullet}$. 
Set  $\ul\bB_{\eta; \ge a} \cap \ul\bB_{\nu} = 
    \bigsqcup_{a' \ge a}(\ul\bB_{\eta;a'} \cap \ul\bB_{\nu})$. 
We show 
\par\medskip\noindent
(4.20.5) \ The elements in $\ul\bB_{\eta; \ge a} \cap \ul\bB_{\nu}$ are almost orthonormal, 
and $\ul\bB_{\eta;\ge a} \cap \ul\bB_{\nu}$ 
gives a $\BZ[q]$-basis of $\ul Z_{\eta; \ge a} \cap \ul\SL_{\BZ}(\infty) 
        \cap (\ul\BU_q^-)_{\nu}$. 

\par\medskip
In fact, take $b \in \ul\bB_{\eta;a} \cap \ul\bB_{\nu}$.  Since 
$b \equiv b_{[\eta; a]} \mod q\ul\SL_{\BZ}(\infty)$ by (4.20.2),  and since 
$b_{[\eta;a]}$ is orthogonal to $\ul\bB_{\eta; >a} \cap \ul\bB_{\nu}$ by (1.12.2), 
$(b, b_0) \in q\BZ[[q]] \cap \BQ(q)$ for any $b_0 \in \ul\bB_{\eta; >a} \cap \ul\bB_{\nu}$. 
On the other hand, it follows from the construction, 
there exists $b'\in \ul\bB_{\eta;0} \cap \ul\bB_{\nu'}$
with $|\nu'| < |\nu|$ such that  
$b_{[\eta;a]} = (\ul f_{\eta}^{(a)}b')_{[\eta;a]}$.
Here $(\ul f_{\eta}^{(a)}b')_{[\eta;a]} 
       = \ul f_{\eta}^{(a)}(b'_{[\eta;0]})$.
Since $b' \in \ul\bB_{\eta;0} \cap \ul\bB_{\nu'}$, we have 
$b' \equiv b'_{[\eta; 0]} \mod q\ul\SL_{\BZ}(\infty)$ by (4.20.2).  
Now take $b_1 \in \ul\bB_{\eta;a} \cap \ul\bB_{\nu}$. There exists
$b_1' \in \ul\bB_{\eta;0} \cap \ul\bB_{\nu'}$ satisfying similar properties 
as in the case of $b$.
Thus we have 
\begin{align*}
(b, b_1) &\equiv (b_{[\eta;a]}, (b_1)_{[\eta;a]})  \mod q\BZ[[q]] \cap \BQ(q) \\ 
         &=  \bigl((\ul f_{\eta}^{(a)}b')_{[\eta;a]}, 
                   (\ul f_{\eta}^{(a)}b_1')_{[\eta;a]}\bigr) \\
         &=  \bigl(\ul f_{\eta}^{(a)}b'_{[\eta:0]}, 
                   \ul f_{\eta}^{(a)}(b_1')_{[\eta;0]}\bigr) \\ 
         &= c(b'_{[\eta;0]}, (b_1')_{[\eta;0]}) \qquad \text{ with }
                 c \in 1 + q(\BZ[[q]] \cap \BQ(q)), 
\end{align*}
where the last equality follows from (1.9.2). 
Since $(b'_{[\eta;0]}, (b_1')_{[\eta;0]}) \equiv (b', b_1')$, 
we have $(b, b_1) \equiv (b', b_1') \mod q(\BZ[[q]] \cap \BQ(q))$. 
Hence the almost orthonormality for $b, b_1$ follows from that for $b', b_1'$. 
Thus we have proved the almost orthonormality for $\ul\bB_{\eta; \ge a} \cap \ul\bB_{\nu}$.
Now by applying Lemma 4.15 (ii) for 
$M = \ul Z_{\eta; \ge a} \cap (\ul\BU_q^-)_{\nu}$, 
$\ul\bB_{\eta; \ge a} \cap \ul\bB_{\nu}$ gives an $\BA$-basis of $M$. 
Since $\vf(\ul\bB_{\eta; \ge a} \cap \ul\bB_{\nu})$ gives an $\BF[q]$-basis
of $\vf(M) \cap \ul\SL_{\BF}(\infty)$,
we see that $\ul\bB_{\eta; \ge a} \cap \ul\bB_{\nu}$ gives a $\BZ[q]$-basis 
of $M \cap \ul\SL_{\BZ}(\infty)$.  Thus (4.20.5) holds.

\par
By the above procedure, one can construct $\ul\bB_{\eta;a} \cap \ul\bB_{\nu}$ 
for any $a > 0$.  But this method cannot be applied for   
constructing $\ul\bB_{\eta;0} \cap \ul\bB_{\nu}$.
In order to treat this case, we prepare a lemma.
\begin{lem}   
Let $a \ge 0$. 
Let $\ul\bB_{\eta; > a} \cap \ul\bB_{\nu}$ be the almost orthonormal basis 
of $\sum_{a' > a}\ul f_{\eta}^{(a')}{}_{\BA}\ul\BU_q^- \cap (\ul\BU_q^-)_{\nu}$ 
constructed in 4.19. 
Assume that $x \in \ul\SL_{\BZ}(\infty)$ with $(x,x) \in 1 + q\BA_0$.
Further assume that 
$x \in \sum_{a' \ge a}\ul f_{\eta}^{(a')}{}_{\BA}\BU_q^- \cap (\ul\BU^-_q)_{\nu}$ 
and that $\vf(x) = b_{\bullet}$ for some $b_{\bullet} \in \ul\bB^{\bullet}_{\eta;a}$. 
Then 
\begin{equation*}
\tag{4.21.1}
x \equiv x_{[\eta;a]} \mod q\ul\SL_{\BZ}(\infty).
\end{equation*}
\end{lem} 

\begin{proof}
By Corollary 4.11, $b_{\bullet} \in \ul\bB^{\bullet}_{\eta;a}$ is written as
$b_{\bullet} = (b_{\bullet})_{[\eta;a]} + \sum_{b'_{\bullet}}a_{b'_{\bullet}}b'_{\bullet}$,
where $b'_{\bullet} \in \ul\bB^{\bullet}_{\eta; > a}$ and $a_{b'_{\bullet}} \in q\BF[q]$. 
Take $b' \in \ul\bB_{\eta; > a} \cap \ul\bB_{\nu}$ such that $\vf(b') = b'_{\bullet}$. 
Note that $\vf(x_{[\eta;a]}) = (b_{\bullet})_{[\eta;a]}$. 
Then $z = x - x_{[\eta;a]} - \sum_{b'}a_{b'}b' \in p({}_{\BA}\ul\BU_q^-)$, where 
$b' \in \ul\bB_{\eta; > a} \cap \ul\bB_{\nu}$ 
and $a_{b'} \in q\BZ[q]$ is such that its image to $q\BF[q]$ coincides with $a_{b'_{\bullet}}$. 
Here $z \in \ul Z_{\eta; > a} \cap \ul\SL_{\BZ}(\infty)$. 
Hence by (4.20.5), $z$ can be written as $z = \sum_{b'}c_{b'}b'$, where 
$b' \in \ul\bB_{\eta; > a} \cap \ul\bB_{\nu}$ and $c_{b'} \in p\BZ[q]$. 
Suppose that $z \notin q\ul\SL_{\BZ}(\infty)$. Then there exists $b'_0$ such that
$c_{b'_0} \in p\BZ[q] - q\BZ[q]$.  This implies that $c_{b'_0} \equiv d \mod q\BZ[q]$
for some $0 \ne d \in p\BZ$. 
We write $x = \sum_{n \ge a}\ul f^{(n)}_{\eta}x_n$ with $x_n \in \ul\BU_q^-[\eta]$. 
Then $\ul f_{\eta}^{(a)}x_a = x_{[\eta;a]}$.  Since $\vf$ is compatible with bilinear forms,  
the image of $(x_{[\eta;a]}, x_{[\eta;a]})$ on $\BF(q)$ coincides with 
$((b_{\bullet})_{[\eta;a]}, (b_{\bullet})_{[\eta;a]})$ which is
contained in $1 + (q\BF[[q]] \cap \BF(q))$.  
It follows that $(x_a, x_a) \notin q\BA_0$.  Then by Lemma 1.11 (i), 
$(x_a, x_a) \in 1 + q\BA_0$, and $(x_{a'}, x_{a'}) \in q\BA_0$ for all $a' \ne a$.    
But by the almost orthonormality of the basis $\ul\bB_{\eta; > a}\cap \ul\bB_{\nu}$, 
if there exists $b'_0$ 
such that $c_{b'_0}  \equiv d \mod q\BZ[q]$ with $d \in p\BZ - \{ 0\}$, we must have
$(x_{a'}, x_{a'}) \notin q\BA_0$ for $a'$ such that 
$b'_0 \in \ul\bB_{\eta;a'}$.  This is absurd, and we conclude that
$z \in q\ul\SL_{\BZ}(\infty)$.  The lemma is proved. 
\end{proof}

\para{4.22.}
We now construct $\ul\bB_{\eta; 0} \cap \ul\bB_{\nu}$.
Take $b_{\bullet} \in \ul\bB^{\bullet}_{\eta:0} \cap \ul\bB^{\bullet}_{\nu}$. 
There exists $\eta' \in \ul I$ such that 
$\ve_{\eta'}(b_{\bullet}) = a > 0$. 
By applying the discussion in 4.19 for $\eta'$ and $a > 0$, we can 
find $b \in \ul\bB_{\eta'; a} \cap \ul\bB_{\nu}$ such that $\vf(b) = b_{\bullet}$. 
We define $\ul\bB_{\eta:0} \cap \ul\bB_{\nu}$ as the set of those $b$ such that 
$\vf(b) = b_{\bullet}$ for various 
$b_{\bullet} \in \ul\bB^{\bullet}_{\eta:0} \cap \ul\bB^{\bullet}_{\nu}$ 
(and for various $\eta'$).
We set $\ul\bB_{\nu} = \bigsqcup_{a \ge 0}(\ul\bB_{\eta: a} \cap \ul\bB_{\nu})$. 
\par 
We know, by (4.20.5), that $\ul\bB_{\eta; > 0} \cap \ul\bB_{\nu}$ gives 
an almost orthonormal basis of 
$\ul Z_{\eta; >0} \cap \ul\SL_{\BZ}(\infty) \cap (\ul\BU_q)_{\nu}$. 
Hence by applying Lemma 4.21 for $a = 0$, we see that, for 
$b \in \ul\bB_{\eta;0} \cap \ul\bB_{\nu}$, 
\begin{equation*}
\tag{4.22.1}
b \equiv b_{[\eta;0]} \mod q\ul\SL_{\BZ}(\infty).
\end{equation*} 

Next we note that
\par\medskip\noindent
(4.22.2) \ The elements in $\ul\bB_{\nu}$ are almost orthonormal, 
and $\ul\bB_{\nu}$ gives a $\BZ[q]$-basis of $\ul\SL_{\BZ}(\infty) \cap (\ul\BU_q^-)_{\nu}$.
\par\medskip
We show (4.22.2).  Take $b \in \ul\bB_{\eta:0}$.  
If $b' \in \ul\bB_{\eta; > 0} \cap \ul\bB_{\nu}$, 
$(b_{[\eta:0]}, b') = 0$.  Hence $(b, b') \in q\BZ[[q]] \cap \BQ(q)$ by (4.21.1). 
So, assume that $b, b' \in \ul\bB_{\eta:0}$.  Set $b_{\bullet} = \vf(b), 
b'_{\bullet} = \vf(b')$, and  
$\ve_{\eta'}(b_{\bullet}) = a, \ve_{\eta'}(b'_{\bullet}) = a'$.
By our assumption, $a > 0$.  If $a' > 0$, then by applying the discussion 
in (4.20.5)  for $\ul\bB_{\eta'; >0}$,  $b$ and $b'$ are almost orthonormal.  
Thus assume that $a' = 0$.  By applying Lemma 4.21 for 
$\ul\bB_{\eta'; >0} \cap \ul\bB_{\nu}$, we have 
$b' \equiv b'_{[\eta'; 0]} \mod q\ul\SL_{\BZ}(\infty)$. 
Since $(b, b'_{[\eta';0]}) = 0$, we have $(b, b') \in q\BZ[[q]] \cap \BQ(q)$.  
Hence $\ul\bB_{\nu}$ is almost orthonormal.    
The second assertion of (4.22.2) is shown as in (4.20.5). 
Thus (4.22.2) holds. 
\par
In the construction of $\ul\bB_{\nu}$, the choice of $\eta, \eta'$, etc. is 
not unique.  But the following lemma shows that $\ul\bB_{\nu}$ is determined 
up to the sign. 

\begin{lem}  
Let $\ul\bB_{\nu}$ be a fixed basis of $\ul\SL_{\BZ}(\infty) \cap (\ul\BU_q^-)_{\nu}$ 
constructed as above.  
\begin{enumerate}
\item \ Let $x \in \ul\SL_{\BZ}(\infty) \cap (\ul\BU_q^-)_{\nu}$ be an element 
such that $\ol x = x$ and that $(x,x) \in 1 + q\BA_0$. Assume further that 
$\vf(x) \in \ul\bB^{\bullet}_{\nu}$.
Then $x = b$ (resp. $x = \pm b$) if $p \ne 2$ (resp. $p = 2$), 
where $b$ is the unique element in $\ul\bB_{\nu}$ such that 
$\vf(x) = \vf(b)$. 
\item \ The set $\ul\bB_{\nu}$ is determined uniquely (resp. uniquely up to sign) 
if $p \ne 2$ (resp. $p = 2$), independent from 
the construction process.  
\end{enumerate}
\end{lem} 

\begin{proof}
We show (i). Set $\vf(x) = b_{\bullet} \in \ul\bB^{\bullet}_{\nu}$, and let 
$b \in \bB_{\nu}$ be such that $\vf(b) = b_{\bullet}$. 
Then $x - b \in \ul\SL_{\BZ}(\infty) \cap (\ul\BU_q^-)_{\nu} \cap p(\ul\BU_q^-)$.   
Hence $x - b$ is written as $x - b = \sum_{b'}a_{b'}b'$, where 
$b' \in \ul\bB_{\nu}$ and $a_{b'} \in \BZ[q] \cap p\BA$. But since $x - b$ is 
bar-invariant, $a_{b'} \in p\BZ$ for any $b' \in \ul\bB_{\nu}$.  Then 
$x = (a_b + 1)b + \sum_{b' \ne b}a_{b'}b'$. Since $\ul\bB_{\nu}$ is almost orthonormal
by (4.22.2), we have
\begin{equation*}
\tag{4.23.1}
(x,x) \equiv (a_b + 1)^2 + \sum_{b' \ne b}a_{b'}^2 \mod q\BA_0. 
\end{equation*}
Since $a_b \in p\BZ$, $a_b + 1 \ne 0$.  Since $(x,x) \equiv 1 \mod q\BA_0$, 
(4.23.1) implies that $a_{b'} = 0$ for any $b' \in \ul\bB_{\nu} - \{ b\}$, hence 
$a_b + 1 = \pm 1$. This implies that $a_b + 1 = 1$ and $x = b$ if $p \ne 2$. 
While if $p = 2$, we have $a_b + 1 = \pm 1$, and so $x = \pm b$.
(i) is proved.  
\par 
We show (ii).  Let $\ul\bB'_{\nu}$ be a set constructed in a similar way 
as $\ul\bB_{\nu}$, but using different $\eta \in \ul I$.  Let $x \in \ul\bB'_{\nu}$. 
Then $x$ satisfies all the conditions in (i).  Assume that $p = 2$. Then by (i), 
$\pm x \in \ul\bB_{\nu}$, 
and so $\pm \ul\bB'_{\nu} \subset \pm \ul\bB_{\nu}$.
Replacing the role of $\ul\bB_{\nu}$ and $\ul\bB_{\nu}$, we have 
$\pm \ul\bB_{\nu} \subset \pm \ul\bB'_{\nu}$. Hence 
$\pm \ul\bB_{\nu} = \pm \ul\bB'_{\nu}$, and (ii) holds.  
The case $p \ne 2$ is similar. 
The lemma is proved. 
\end{proof}

\para{4.24.}
In the case where $p \ne 2$, by Lemma 4.23, $\ul\bB_{\nu}$ is defined canonically, 
independent of the construction process in 4.20. 
While if $p = 2$, Lemma 4.23 determines $\bB_{\nu}$ only up to sign. 
Now assume that $p \ne 2$. 
We define $\ul\bB = \bigsqcup_{\nu \in Q_-}\ul\bB_{\nu}$. 
We show that $\ul\bB$ is the canonical basis of $\ul\BU_q^-$.
Clearly $\bB$ satisfies the properties (C1) $\sim$ (C4) in 1.12.
Since $\vf$ gives a bijection $\vf : \ul\bB_{\nu} \isom \ul\bB_{\nu}^{\bullet}$
for any $\nu \in Q_-$, it gives a bijection $\vf : \bB \isom \ul\bB^{\bullet}$.  
(C5) follows from Lemma 4.21. 
Next we show (C7).  
Take $b \in \bB_{\eta;0}$ for $\eta \in \ul I$. If $\eta$ is the one used for 
the construction of $\bB$, (4.20.3) holds. Since we can choose any $\eta \in \ul I$ for 
the construction by the previous remark, (4.20.3) holds for this $\eta$. 
By (4.12.3), $b_{\bullet} \mapsto b'_{\bullet}$ gives a bijection 
$\ul\bB^{\bullet}_{\eta;0} \to \ul\bB^{\bullet}_{\eta;a}$. 
Hence $b \mapsto b'$ gives a bijection $\ul\bB_{\eta;0} \isom \ul\bB_{\eta;a}$.
This proves (C7). (C6) follows from the corresponding property for $\ul\bB^{\bullet}$
(see 4.20) by using the bijection $\vf : \ul\bB \isom \ul\bB^{\bullet}$.  
Thus $\ul\bB$ is the canonical basis of $\ul\BU_q^-$.
\par
Now assume that $p = 2$.  The above discussion shows that, 
there exists a basis $\ul\CB$ such that $\wt{\ul\CB} = \ul\CB \sqcup -\ul\CB$, 
satisfying the properties (C1) $\sim$ (C6), and (C7$'$).  

\para{4.25.}
We now prove Theorem 4.18.  (i) is already shown. 
$*$-operation commutes with $\vf : {}_{\BA}\ul\BU_q^- \to {}_{\BA'}\ul\BU_q^-$.
By (4.12.4), we have $*(\ul\bB^{\bullet}) = \ul\bB^{\bullet}$.
Since $\vf$ induces a bijection $\ul\bB \isom \ul\bB^{\bullet}$, we obtain 
$*(\ul\bB) = \ul\bB$.  Hence (iii) holds. 
We show (ii). The bijection $\vf : \ul\bB \isom \ul\bB^{\bullet}$ 
induces a bijection 
\begin{equation*}
\xymatrix{
  \xi : \bB^{\s} \ar[rr]^{\pi} &   &  
            \bB^{\dia} \ar[rr]^{\Phi\iv} &  &  
                \ul\bB^{\bullet} \ar[rr]^{\vf\iv} &  &  \ul\bB. 
}
\end{equation*}
Kashiwara operators $\wt F_{\eta} : \bB^{\s} \to \bB^{\s}$ are obtained as the 
restriction of $\prod_{i \in \eta}F_i$ on $\bB^{\s}$. Hence by Proposition 1.19, 
for any $b \in \bB^{\s}$, there exists a sequence $\eta_1, \dots, \eta_N \in \ul I$,
and $c_1, \dots, c_N \in \BZ_{>0}$ such that 
\begin{equation*}
\tag{4.25.1}
b = \wt F^{c_1}_{\eta_1}\wt F^{c_2}_{\eta_2}\cdots \wt F^{c_N}_{\eta_N}1.
\end{equation*}   
Kashiwara operators $F^{\bullet}_{\eta}$ on $\ul\bB^{\bullet}$ are defined by using 
the bijection  $\Phi\iv \circ \pi : \bB^{\s} \to \bB^{\dia} \to \ul\bB^{\bullet}$, 
and a similar property as (4.25.1) holds for $\ul\bB^{\bullet}$. 
Now Kashiwara operators $\ul F_{\eta}$ on $\ul\bB$ are defined for the canonical 
basis $\ul\bB$ of $\ul\BU_q^-$, and from the construction, $\vf$ is compatible with 
those Kashiwara operators. 
Hence $\xi : \bB^{\s} \isom \ul\bB$ is compatible with Kashiwara operators 
$\wt F_{\eta}$ and $\ul F_{\eta}$. Note that by Proposition 1.19, 
$\ul\bB$ also satisfies a similar formula as (4.25.1) by replacing $\wt F_{\eta}$
by $\ul F_{\eta}$. Thus such a bijection $\xi : \bB^{\s} \isom \ul\bB$ is unique. 
Hence (ii) holds. The proof of Theorem 4.18 is now complete.

\para{4.26.}
We return to the general setup, and let $\s$ be an admissible 
diagram automorphism of $X$ of any order.  Let $\ul X$ be the Cartan 
datum induced from $(X, \s)$.  Let $\BU_q^-$ (resp. $\ul\BU_q^-$) be
the quantum enveloping algebra associated to $X$ (resp. $\ul X$). 
We assume that $X$ is a symmetric Cartan datum.  
Then by Theorem 1.18, there exists the canonical basis $\bB$ for 
$\BU_q^-$, which satisfies the property that $*(\bB) = \bB$.  
\par
The following result is our main theorem. 

\begin{thm}  
Under the notation in 4.26, let $\bB$ be the canonical basis of $\BU_q^-$. 
\begin{enumerate}
\item 
Assume that the order of $\s$ is odd. 
Then there exists the canonical basis $\ul\bB$ of $\ul\BU_q^-$, and 
a unique bijection $\xi: \bB^{\s} \isom \ul\bB$ which is compatible with 
Kashiwara operators. Moreover, we have $*(\ul\bB) = \ul\bB$. 
\item 
Assume that the order of $\s$ is even. Then there exists a basis $\ul\CB$ of $\ul\BU_q^-$
satisfying the properties as in Remark 4.19, and a bijection 
$\xi : \wt\bB^{\s} \isom \wt{\ul\CB}$ compatible with Kashiwara operators 
(up to sign).   
\end{enumerate}
\end{thm}

\begin{proof}
We prove the theorem in the case where the order of $\s$ is odd.  
The even case is similar. 
By Proposition 2.5, there exists a sequence 
$X = X_0, X_1, \dots, X_k  = \ul X$, and a diagram automorphism 
$\s_i : X_i \to X_i$ such that $X_{i+1}$ is isomorphic to the Cartan datum 
induced from $(X_i, \s_i)$ and that $\s = \s_{k-1}\cdots \s_1\s_0$. 
Moreover, the order of $\s_i$ is an odd prime power. 
Let ${}_{(i)}\BU_q^-$ be the quantum  algebra associated to $X_i$. 
By induction on $i$, we may assume that ${}_{(i)}\BU_q^-$ has 
the canonical basis ${}_{(i)}\bB$, which is stable by the 
$*$-operation. Then by Theorem 4.18, there exists 
the canonical basis ${}_{(i+1)}\bB$ of ${}_{(i+1)}\BU_q^-$, which is stable by 
the $*$-operation. and 
a bijection $\xi_i : ({}_{(i)}\bB)^{\s_i} \isom {}_{(i+1)}\bB$, compatible with 
Kashiwara operators.  
Thus we obtain the canonical basis $\ul\bB = {}_{(k)}\bB$ of 
$\ul\BU_q^- = {}_{(k)}\BU_q^-$, which is stable by the $*$-operation.
From the construction, we have a bijection  
\begin{equation*}
\xi : \bB^{\s} = (\cdots (\bB^{\s_0})^{\s_1}\cdots )^{\s_{k-1}} \isom \ul\bB
\end{equation*}  
as the composite of $\xi_0, \xi_1, \dots, \xi_{k-1}$.
Each $\xi_i$ is compatible with Kashiwara operators, hence $\xi$ is compatible
with Kashiwara operators.  In particular, $\xi$ is uniquely determined, independent 
of the expression $\s = \s_{k-1}\cdots \s_1\s_0$.  The theorem is proved.
\end{proof}

\remark{4.28.}
\ Recall, by 2.2, that for a Cartan datum $X$ of arbitrary type, 
there exists a symmetric Cartan datum $\wt X$, and an admissible diagram 
automorphism  $\s$ such that the Cartan datum induced from $(\wt X, \s)$ is 
isomorphic to $X$. Thus Theorem 4.27 assures that the quantum enveloping 
algebra $\BU_q^-$ of any type has the canonical (signed) basis $\bB$ in the sense 
of 1.12 and Remark 4.19, which is stable by the $*$-operation.

\par\bigskip

\section{ The proof of Proposition 3.5 }

\para{5.1.}
We follow the notation in Section 3. Note that
${}_{\BA'}\ul\BU_q^-$ is the $\BA'$-algebra with generators 
$\ul f^{(a)}_{\eta}$ ($ \eta \in \ul I, a \in \BN$) with fundamental
relations

\begin{align*}
\tag{5.1.1}
&\sum_{k = 0}^{1 - a_{\eta\eta'}}(-1)^k
        \ul f_{\eta}^{(k)}\ul f_{\eta'}\ul f_{\eta}^{(1 - a_{\eta\eta'} - k)} = 0, 
             \quad (\eta \ne \eta'), \\
\tag{5.1.2}
&[a]_{d_{\eta}}^!\ul f_{\eta}^{(a)} = \ul f_{\eta}^a, \quad (a \in \BN),
\end{align*}
where $d_{\eta} = (\a_{\eta}, \a_{\eta})_1/2$. 
In order to prove Proposition 3.5, it is enough to show that 
$\wt f_{\eta}$ satisfies the following relations in ${}_{\BA'}\BU_q^{-,\s}$, 
\begin{align*}
\tag{5.1.3}
&\sum_{k= 0}^{1 - a_{\eta\eta'}}(-1)^k\begin{bmatrix}
                                        1 - a_{\eta\eta'} \\
                                             k
                                      \end{bmatrix}_{d_{\eta}} 
       \wt f_{\eta}^k \wt f_{\eta'}\wt f_{\eta}^{1 - a_{\eta\eta'} - k} \equiv 0 \mod J, 
          \qquad (\eta \ne \eta'), \\ 
\tag{5.1.4}
&[a]^!_{d_{\eta}}\wt f_{\eta}^{(a)} = \wt f_{\eta}^a, 
          \qquad (a \in \BZ_{\ge 0}).
\end{align*} 

(5.1.4) is shown as follows.
Since $|\eta|$ is a power of $p$, we have $([a]_{d_i}^!)^{|\eta|} = [a]^!_{|\eta|d_i}$ 
in $\BA' = \BF[q,q\iv]$.  Since $d_{\eta} = |\eta|d_i$ for 
$i \in \eta$, we have 
\begin{equation*}
\wt f^{(a)}_{\eta} = \prod_{i \in \eta}f_i^{(a)} 
                   = ([a]_{d_i}^!)^{-|\eta|}\prod_{i \in \eta}f_i^a 
                   = ([a]^!_{d_{\eta}})\iv \wt f_{\eta}^a.
\end{equation*} 
Hence (5.1.4) holds.  The rest of this section is devoted to the 
proof of (5.1.3). 
\par
Recall that the following Serre relations hold in $\BU_q^-$.
For $i \ne j \in I$, 
\begin{equation*}
\tag{5.1.5}
\sum_{k = 0}^{1 - a_{ij}}(-1)^k\begin{bmatrix}
                                    1 - a_{ij} \\
                                        k
                               \end{bmatrix}_{d_i}f_i^kf_jf_i^{1-a_{ij}-k} = 0.
\end{equation*}

We fix $\eta \ne \eta' \in \ul I$, and write them as
$\eta = \{ 1_1, \dots, 1_m \}, \eta' = \{ 2_1, \dots, 2_{n-1}\}$, here 
$|\eta| = m, |\eta'| = n-1$. 

\para{5.2.}
Here we consider the special case where 
$|\eta| = 1$, 
and any $j \in \eta'$ is joined to $i = 1 \in \eta$.  Thus $a_{ij}$ is independent 
of the choice of $j \in \eta'$, and we set $r = 1 - a_{ij}$.
In this case, we have $-a_{\eta\eta'} = |\eta'|(-a_{ij}) = (n-1)(r-1)$.
We set
\begin{equation*}
\tag{5.2.1}
L = 1 - a_{\eta\eta'} = (n-1)(r-1) + 1.
\end{equation*} 
We have $\wt f_{\eta} = f_1, \wt f_{\eta'} = f_{2_1}\cdots f_{2_{n-1}}$.     
In order to verify the formula (5.1.3), we need to compute 
$\wt f_{\eta}^k\wt f_{\eta'}\wt f_{\eta}^{L - k} = f_1^kf_{2_1}\cdots f_{2_{n-1}}f_1^{L-k}$ 
for various $0 \le k \le L$. 
More generally, for each tuple $(a_1, \dots, a_n) \in \BN^n$ 
such that $\sum _i a_i = L$, we consider the correspondence 
\begin{equation*}
\tag{5.2.2}
(a_1, \dots, a_n) \longleftrightarrow 
f_1^{a_1}f_{2_1}f_1^{a_2}f_{2_2}\cdots f_{2_{n-2}}f_1^{a_{n-1}}f_{2_{n-1}}f_1^{a_n} 
                 \in \BU_q^-. 
\end{equation*}
The commuting relations for $f_1$ and $f_{2_k}$ are given by the 
Serre relations (5.1.5) 
for $r = 1 - a_{ij}$ with $i = 1, j = 2_k$.  

\para{5.3.}
Based on the observation in 5.2, we introduce the following 
combinatorial object.
We fix $r \ge 2$.
Let $V_n$ be a vector space over $\BQ(q)$ spanned by $a = (a_1, \dots, a_n) \in \BN^n$, 
which satisfies the following relations for $1 \le i \le n - 1$;
if $a_i \ge r$ for some $i$, then $a$ is written as 
\begin{equation*}
\tag{5.3.1}
\sum_{0 \le j \le r}(-1)^j\begin{bmatrix}
                      r  \\
                      j
                    \end{bmatrix}a(j) = 0, 
\end{equation*}
where $a(j) \in \BN^n$ is given by 
\begin{equation*}
\tag{5.3.2}
a(j) = (a_1, \dots, a_{i-1}, a_i-j, a_i + j, a_{i+2}, \dots, a_n). 
\end{equation*}
In particular, $a(0) = a$. 
For each $m \ge 1$, we denote by $E_n(m)$ the subspace of $V_n$ 
spanned by 
\begin{equation*}
\SE_n(m) = \{ a = (a_1, \dots, a_n) \mid \sum_{1 \le i \le n}a_i = m,
                 a_i \in [0, r-1] \text{ for } 1 \le i < n\}. 
\end{equation*}
It is clear if $a = (a_1, \dots, a_n) \in V_n$ is such that 
$\sum a_i = m$, then $a \in E_n(m)$.
\par
We prove the following. 

\begin{prop}   
Assume that $n = 2$.  Then for any $k \ge 0$, the following 
formula holds in $V_2$. 
\begin{equation*}
\tag{5.4.1}
(r + k,\ell) = \sum_{s=1}^r(-1)^{s-1}
                  \begin{bmatrix}
                     s + k - 1 \\
                        s -1    
                  \end{bmatrix}
                 \begin{bmatrix}
                    r + k  \\
                    r - s
                 \end{bmatrix}(r - s, \ell + s + k).
\end{equation*}
\end{prop}

\para{5.5.}
(5.4.1) holds for $k = 0$ by (5.3.1). 
We shall prove (5.4.1) by induction on $k$. 
Assume that (5.4.1) holds for $k' < k$.  

\par
First assume that $k \le r$.  In this case, we have

\begin{align*}
(r &+ k, \ell)  \\ 
   &= \sum_{s = 1}^k(-1)^{s-1}
                     \begin{bmatrix}
                       r  \\
                       s
                     \end{bmatrix}(r + k -s, \ell + s)
      + \sum_{s = k+1}^r(-1)^{s-1}
                     \begin{bmatrix}
                       r  \\
                       s
                     \end{bmatrix}(r + k -s, \ell + s)   \\
   &= \sum_{s=1}^k(-1)^{s-1}\begin{bmatrix}
                       r  \\
                       s
                     \end{bmatrix}
                  \biggl(\sum_{t= 1}^r(-1)^{t-1}
                      \begin{bmatrix}
                        t + (k - s)-1 \\
                           t - 1
                      \end{bmatrix}
                      \begin{bmatrix}
                         r + (k - s) \\
                         r - t
                      \end{bmatrix}(r - t, \ell + t + k)\biggl)  \\
     &\qquad  + \sum_{t = 1}^{r - k}(-1)^{t + k -1}
                     \begin{bmatrix}
                       r  \\
                       t + k
                     \end{bmatrix}(r - t, \ell + t + k)  \\
     &= \sum_{t = 1}^r(-1)^{t-1}A_t(r-t, \ell + t + k), 
\end{align*}
where

\begin{align*}
\tag{5.5.1}
A_t = \biggl(\sum_{s=1}^k(-1)^{s-1}
              \begin{bmatrix}
                  r  \\
                  s
              \end{bmatrix}
              \begin{bmatrix}
                  t + k - s -1 \\
                  t - 1
              \end{bmatrix}
              \begin{bmatrix}
                  r + k - s \\
                  r - t
              \end{bmatrix}\biggr) + (-1)^k\begin{bmatrix}
                       r  \\
                       t + k
                     \end{bmatrix}. 
\end{align*}
The last term appears only in the case where $t + k \le r$.   
A similar computation shows, in the case where $k > r$, we have
\begin{equation*}
\tag{5.5.2}
A_t = \sum_{s=1}^r(-1)^{s-1}
              \begin{bmatrix}
                  r  \\
                  s
              \end{bmatrix}
              \begin{bmatrix}
                  t + k - s -1 \\
                  t - 1
              \end{bmatrix}
              \begin{bmatrix}
                  r + k - s \\
                  r - t
              \end{bmatrix}.  
\end{equation*}

We prove the following formula. 

\begin{lem}  
Assume that $1 \le t \le r$ and $k \in \BZ$.
Then we have
\begin{equation*}
\tag{5.6.1}
\sum_{s = 0}^r(-1)^s\begin{bmatrix}
                       r \\
                       s
                    \end{bmatrix}
                    \begin{bmatrix}
                  t + k - s -1 \\
                  t - 1
                   \end{bmatrix}
                   \begin{bmatrix}
                  r + k - s \\
                  r - t
                   \end{bmatrix} = 0. 
\end{equation*}
\end{lem}

\begin{proof}
Note that, for any $k \in \BZ$, one can write as 
\begin{align*}
\tag{5.6.2}
[k - s][k - s -1]\cdots [k - s - r + 1]
   = \sum_{j = 0}^{r}F_j(q)q^{sr - 2js},
\end{align*}
where $F_j(q) \in \BQ(q)$ is independent from $s$. 
Applying (5.6.2) to our situation, we have 
\begin{align*}
\begin{bmatrix}
                  t + k - s -1 \\
                      t -1 
                   \end{bmatrix} &= \sum_{j = 0}^{t-1}F_j(q)q^{s(t-1) - 2js}, \\
                    \begin{bmatrix}
                  r + k - s \\
                  r - t
                   \end{bmatrix} &= \sum_{j= 0}^{r-t}G_j(q)q^{s(r-t)-2js}, 
\end{align*}
where $F_j(q), G_j(q) \in \BQ(q)$ are independent from $s$. It follows that 
\begin{equation*}
\tag{5.6.3}
                 \begin{bmatrix}
                  t + k - s -1 \\
                      t -1 
                   \end{bmatrix}
                   \begin{bmatrix}
                  r + k - s \\
                  r - t
                   \end{bmatrix} = \sum_{j = 0}^{r-1}H_j(q)q^{s(r-1 - 2j)},
\end{equation*}
where $H_j(q) = \sum_{j' + j'' = j}F_{j'}(q)G_{j''}(q)$ is independent from $s$.

Recall the formula (1.11.5), 
\begin{equation*}
\tag{5.6.4}
\prod_{\ell =0}^{n-1}(1 + q^{2\ell}x) = \sum_{k = 0}^nq^{k(n-1)}
                  \begin{bmatrix}
                      n   \\
                      k
                  \end{bmatrix}x^k,
\end{equation*}
where $x$ is another indeterminate.  If we put $x = -q^{-2j}$ 
for $j = 0, \dots, n - 1$, the left hand side of (5.6.4) is equal to zero.
Hence
\begin{equation*}
\tag{5.6.5}
\sum_{k = 0}^n(-1)^kq^{k(n -1 -2j)}\begin{bmatrix}
                       n  \\
                       k
                    \end{bmatrix} = 0  \qquad\text{ for } j = 0, \dots, n - 1.
\end{equation*}
If we substitute (5.6.3) into the left hand side of (5.6.1),
then the equality of (5.6.1) follows from (5.6.5). 
Thus the lemma is proved. 
\end{proof}

\para{5.7.}
We consider the expansion of $(r + k, \ell)$ in 5.5. 
Now assume that $k > r$.  Then by (5.5.2) and Lemma 5.6, 
we see that 
$A_t = \begin{bmatrix}
                                  t + k - 1 \\
                                  t -1
                               \end{bmatrix}
                               \begin{bmatrix}
                                  r + k \\
                                  r - t
                               \end{bmatrix}$.
Hence (5.4.1) holds for $k > r$. 
\par
Next assume that $k \le r$.  In this case, $A_t$ is given by (5.5.1).
We consider the formula in Lemma 5.6.   We fix $s$ such that 
$k + 1 \le s \le r$.  In this case, 
$t + k - s-1 < t - 1$. Hence if $t + k - s - 1 \ge 0$, then 
$\begin{bmatrix}
                  t + k - s -1 \\
                      t -1 
\end{bmatrix} = 0$. 
On the other hand, since $r + k - s \ge 0$, if
$r + k - s < r - t$, then  
$\begin{bmatrix}
                  r + k - s \\
                  r - t
\end{bmatrix} = 0$.
It follows that 
\begin{equation*}
\tag{5.7.1}
                 \begin{bmatrix}
                  t + k - s -1 \\
                      t -1 
                   \end{bmatrix}
                   \begin{bmatrix}
                  r + k - s \\
                  r - t
                   \end{bmatrix} \ne 0
\end{equation*}
only when $t + k - s- 1 < 0$ and $r + k - s \ge r- t$, 
namely only when $s = t + k$.                                   
If $s = t + k$, the left hand side of (5.7.1) is equal to
$\begin{bmatrix}
         -1  \\
         t-1
 \end{bmatrix} = (-1)^{t-1}$.
Hence by Lemma 5.6, we have
\begin{equation*}
\sum_{s = 0}^k(-1)^s\begin{bmatrix}
                       r \\
                       s
                    \end{bmatrix}
                    \begin{bmatrix}
                  t + k - s -1 \\
                  t - 1
                   \end{bmatrix}
                   \begin{bmatrix}
                  r + k - s \\
                  r - t
                   \end{bmatrix}  
     + (-1)^{t + k}\begin{bmatrix}
                       r  \\
                      t + k
                   \end{bmatrix} (-1)^{t-1} = 0.
\end{equation*}
By (5.5.1), we obtain $A_t = \begin{bmatrix}
                              t + k -1 \\
                              t - 1
                           \end{bmatrix}
                           \begin{bmatrix}
                              r + k \\
                              r - t
                           \end{bmatrix}$.
Thus (5.4.1) holds also for $k \le r$.    
The proof of Proposition 5.4 is now complete. 

\para{5.8.}
Proposition 5.4 holds for $k \ge 0$. 
Now assume that $-r \le k < 0$. 
If $s + k -1 \ge 0$, then $\begin{bmatrix}
                     s + k - 1 \\
                       s - 1    
                  \end{bmatrix} = 0$, and 
if $r + k < r-s$, then $\begin{bmatrix}
                    r + k  \\
                    r - s
                 \end{bmatrix} = 0$.   Hence 
                  $A_s = \begin{bmatrix}
                     s + k - 1 \\
                       s - 1    
                  \end{bmatrix}
                 \begin{bmatrix}
                    r + k  \\
                    r - s
                 \end{bmatrix} \ne 0$ only when 
$s + k -1 < 0$ and $r + k \ge r-s$, namely when $s = -k$.
In that case, we have $A_s = \begin{bmatrix}
                                -1  \\
                               s - 1
                             \end{bmatrix} = (-1)^{s-1}$. 
Thus $(r + k, \ell)$ has an expansion as in (5.4.1), 
where all the coefficients of $(r-s, \ell + s + k)$ are zero unless 
$s = -k$, in which case, the coefficient is equal to 1. 
Hence (5.4.1) holds for $r + k \ge 0$. 
Set, for $1 \le s \le r$, 
\begin{equation*}
\tag{5.8.1}
A(k,s) = \begin{bmatrix}
                     s + k - r - 1 \\
                       s - 1    
                  \end{bmatrix}
                 \begin{bmatrix}
                      k  \\
                    r - s
                 \end{bmatrix}.
\end{equation*}
Then by replacing $r + k$ by $k$ in (5.4.1), we have

\begin{cor}   
Assume that $n = 2$, and $k \ge 0$.  Then the following formula holds in $V_2$.
\begin{equation*}
\tag{5.9.1}
(k,\ell) = \sum_{s=1}^r(-1)^{s-1}
                       A(k,s)(r - s, \ell + s + k - r). 
\end{equation*}
\end{cor}

By making use of Corollary 5.9, we obtain a formula for $V_n$ 
for any $n \ge 2$. 

\begin{lem}  
Assume that $(k, 0, \dots, 0, \ell) \in V_n$.  Then we have
\begin{align*}
(k, &0, \dots, 0, \ell)  \\ 
  &= \sum_{\substack{a_1 + \cdots + a_n = k+\ell \\
                         a_1, \dots, a_{n-1} \in [0, r-1]}}
                          (-1)^{a_1 + \cdots + a_{n-1} + (n-1)(r-1)}
      \biggl(\prod_{1 \le i \le n-1}A(k - x_i, r - a_i)\biggr)(a_1, a_2, \dots, a_n),
\end{align*}
where $x_i = a_1 + a_2 + \cdots + a_{i-1}$, and 
$(a_1, \dots, a_n) \in \SE_n(k + \ell)$.  
\end{lem}

\begin{proof}
We prove the lemma by induction on $n$.
By (5.9.1), the lemma holds for $n = 2$. 
We assume that it holds for $n' < n$. 
By applying (5.9.1) for the first two terms $(k,0)$ in 
$(k,0,\dots,0,\ell)$, we have 
 
\begin{align*}
(k, 0, \dots, 0, \ell) 
   &= \sum_{s=1}^r(-1)^{s-1}A(k,s)(r-s, s + k - r, 0, \dots, 0, \ell)   \\
   &= \sum_{a_1 \in [0,r-1]}(-1)^{a_1 + r -1}A(k, r - a_1) (a_1, k - a_1, 0, \dots, 0, \ell).
\end{align*}
Then by applying the induction hypothesis for $(k - a_1, 0, \dots, 0, \ell) \in V_{n-1}$, 
we obtain the required formula. 
\end{proof}

We prove the following formula.

\begin{prop}  
Set $L = (n-1)(r-1) +1$. Then the following equality holds in $E_n(L)$. 

\begin{equation*}
\tag{5.11.1}
\sum_{k = 0}^L(-1)^k\begin{bmatrix}
                        L   \\
                        k
                    \end{bmatrix}(k, 0, \dots, 0, L - k) = 0.
\end{equation*}
\end{prop}

\begin{proof}
We apply Lemma 5.10 to the case where $k + \ell = L$. 
In order to prove (5.11.1), it is enough to see, for a fixed 
$(a_1, \dots, a_n) \in \SE_n(L)$, that
\begin{equation*}
\tag{5.11.2}
\sum_{k = 0}^L(-1)^k\begin{bmatrix}
                        L  \\
                        k
                    \end{bmatrix}
       \biggl(\prod_{1 \le i \le n-1}A(k-x_i, r-a_i)\biggr) = 0. 
\end{equation*}

As in (5.6.3), but by replacing the role of $s$ by $-k$, we see that 
\begin{equation*}
A(k-x_i, r-a_i) = \begin{bmatrix}
                     k - x_i - a_i - 1 \\
                       r - a_i - 1    
                  \end{bmatrix}
                 \begin{bmatrix}
                      k -x_i \\
                        a_i
                 \end{bmatrix} = \sum_{j = 0}^{r-1}F_j(q)q^{-k(r-1 - 2j)}
\end{equation*}
for some $F_j(q) \in \BQ(q)$, which is independent from $k$. 
It follows that 
\begin{equation*}
\tag{5.11.3}
\prod_{1 \le i \le n-1}A(k - x_i, r - a_i) 
        = \sum_{j = 0}^{(n-1)(r-1)}G_j(q)q^{-k\bigl((n-1)(r-1) - 2j\bigr)},
\end{equation*}
where $G_j(q) \in \BQ(q)$ is independent from $k$. 
Note that by (5.6.5), we have 
\begin{equation*}
\sum_{k= 0}^L (-1)^k\begin{bmatrix}
                       L  \\
                       k
                    \end{bmatrix}q^{-k(L -1-2j)} = 0
\end{equation*}
for $j = 0, \dots, L-1$.
Since $L = (n-1)(r-1)+1$, (5.11.3) implies (5.11.2).  The proposition is 
proved. 
\end{proof}

\para{5.12.}
Returning to the setup in 5.2, we consider $\eta, \eta' \in \ul I$
with $|\eta| = 1, |\eta'| = n - 1$, where any element $j \in \eta'$ is joined 
to $i \in \eta$.  Set $r = 1 - a_{ij}$, and $L = 1 - a_{\eta\eta'} = (n-1)(r-1) + 1$.
Set $\wt f_{\eta} = f_1$ and $\wt f_{\eta'} = f_{2_1}\cdots f_{2_{n-1}}$. 
Under the correspondence in (5.2.2), 
$(k, 0, \dots, 0, L -k) \in \SE_n(L)$ corresponds to 
$\wt f_{\eta}^k\wt f_{\eta'}\wt f_{\eta}^{L-k} \in \BU_q^-$. Thus
in view of the discussion in 5.2, Proposition 5.11 can be translated to 
the following formula (by replacing $q$ by $q^{d_i} = q^{d_{\eta}}$), which 
proves (5.1.3) in this special case.  
Note that this formula holds without the assumption on the order of 
$\s$, nor modulo $J$.  
\begin{prop}  
Under the notation above, 
\begin{equation*}
\sum_{k = 0}^{1 - a_{\eta\eta'}}(-1)^k\begin{bmatrix}
                                        1 - a_{\eta\eta'} \\
                                          k
                                      \end{bmatrix}_{d_{\eta}} 
                 \wt f_{\eta}^k\wt f_{\eta'} \wt f_{\eta}^{1-a_{\eta\eta'} - k} = 0.  
\end{equation*}
\end{prop} 

\para{5.14.}
We shall extend Proposition 5.13 to the general setup as in 5.1.
Hence we consider $\eta \ne \eta' \in \ul I$ with $|\eta| = n-1, |\eta'| = m$, 
and write $\eta = \{ 1_1, \dots, 1_m\}, \eta' = \{ 2_1 \dots, 2_{n-1}\}$ 
as in 5.1.
For each $i \in \eta$, 
let $A_i$ be the set of elements in $\eta'$ which is joined to $i$.  Set  
$|A_i| = N -1$, which is independent of $i \in \eta$.  
For $i \in \eta, j \in A_i$, $a_{ij}$ is independent of the choice of $(i,j)$.
We set $r = 1 - a_{ij}$. By (2.1.2), we have
\begin{equation*}
a_{\eta\eta'} = \sum_{j \in A_i}\frac{2(\a_i,\a_j)}{(\a_i,\a_i)}
              = \sum_{j \in A_i}a_{ij} = -(N - 1)(r-1).
\end{equation*}    
We set
\begin{equation*}
\tag{5.14.1}
L = 1 - a_{\eta\eta'} = (N - 1)(r-1) + 1.
\end{equation*}
We have $\wt f_{\eta} = f_{1_1}\cdots f_{1_m}$, $\wt f_{\eta'} = f_{2_1}\cdots f_{2_{n-1}}$.
In this case, we need to compute 
$\wt f_{\eta}^k\wt f_{\eta'}\wt f_{\eta}^{L - k}
   = (f_{1_1}\cdots f_{1_m})^{k}f_{2_1}\cdots f_{2_{n-1}}(f_{1_1}\cdots f_{1_m})^{L-k}$. 
More generally, we consider
\begin{equation*}
\tag{5.14.2}
f_{\eta}^{\Ba^{(1)}}f_{2_1}f_{\eta}^{\Ba^{(2)}}f_{2_2} \cdots 
      f_{2_{n-2}}f_{\eta}^{\Ba^{(n-1)}}f_{2_{n-1}}f_{\eta}^{\Ba^{(n)}} \in \BU_q^- 
\end{equation*}
for $\Ba^{(k)} \in \BN^m$ ($1 \le k \le n$), where we set 
$f_{\eta}^{\Ba} = f_{1_1}^{a_1}\cdots f_{1_m}^{a_m}$ for $\Ba = (a_1, \dots, a_m) \in \BN^m$.  
Corresponding to those elements, we consider the matrix 
$\Ba = (\Ba^{(1)}, \dots, \Ba^{(n)}) = (a^{(k)}_i) \in M(m,n)$ 
(here we regard $\Ba^{(k)} = {}^t(a^{(k)}_1, \dots, a^{(k)}_m)$ as the $k$-th column vector, 
and $\Ba_i = (a^{(1)}_i, \dots, a^{(n)}_i)$ as the $i$-th row vector).  
The commuting relations for $f_i$ and $f_j$ $(i \in \eta, j \in \eta'$) are 
given by (5.1.5) if $j \in A_i$, and $f_if_j = f_jf_i$ otherwise. 

\para{5.15.}
Taking the discussion in 5.14 into account, 
we generalize the combinatorial setting in 5.3 as follows. 
Fix $m,n$ and $N \le n$, and consider $M(m,n)$ as in 5.15. 
For each $i \in [1,m]$, we fix a subset $A_i \subset [1, n-1]$ such that
$|A_i| = N -1$. Let $V_{n,m}$ be a vector space over $\BQ(q)$ spanned by 
$\Ba = (a^{(k)}_i) \in M(m,n)$ satisfying the following relations;
for each $a^{(k)}_i$, 
there exist $\Ba(t)$ with $\Ba(t) = (a(t)^{(k)}_i)$ for $1 \le t \le  r$ 
such that 
\begin{equation*}
\tag{5.15.1}
\Ba = \begin{cases}
          \sum_{1 \le t \le r}(-1)^{t-1}\begin{bmatrix}
                                           r  \\
                                           t
                                        \end{bmatrix}\Ba(t) 
                 &\quad\text{ if $k \in A_i$ and $a^{(k)}_i \ge r$,} \\
         \Ba(1)  &\quad\text{ if $k \notin A_i$ and $a^{(k)}_i \ge 1$, }  
      \end{cases}
\end{equation*}
where 
\begin{equation*}
\tag{5.15.2}
\Ba(t)_i = (a^{(1)}_i, \dots, a^{(k-1)}_i, 
       a^{(k)}_i - t, a^{(k+1)}_i + t, a^{(k+2)}_i, \dots, a^{(n)}_i)
\end{equation*} 
and $\Ba(t)_{i'} = \Ba_{i'}$ for $i' \ne i$.  
\par
For each $\Bs = (s_1, \dots, s_m) \in \BN^m$, we denote by $E_{n,m}(\Bs)$ 
the subspace of $V_{n,m}$ spanned by 
\begin{equation*}
\SE_{n,m}(\Bs) = \{ \Ba = (a^{(k)}_i) \in M(m,n) \mid 
             \sum_{1 \le k \le n}a^{(k)}_i = s_i, a^{(k)}_i \in [0, r-1]
                       \text{ for } 1 \le k < n \} 
\end{equation*} 
In the case where $\Bs = (L', \dots, L') \in \BN^m$ for some $L' \ge 1$, we set
$\SE_{n,m}(\Bs) = \SE_{n,m}(L')$ and $E_{n,m}(\Bs) = E_{n,m}(L')$.

\par
The following lemma is a generalization of Lemma 5.10.  The proof is 
essentially reduced to the case where $m = 1$, which corresponds to
Lemma 5.10.

\begin{lem}   
Assume that $(\Bk, \bold 0, \dots, \bold 0, \Bell) \in V_{n,m}$, 
where $\Bk = {}^t(k_1, \dots, k_m), 
\Bell = {}^t(\ell_1, \dots, \ell_m) \in \BN^m$ are column vectors.
Then we have

\begin{align*}
\tag{5.16.1}
(\Bk, \bold 0, \dots, \bold 0, \Bell) = \sum_{\Ba \in \SE_{n,m}(\Bk + \Bell)}
            \biggl(\prod_{1 \le i \le m}(-1)^{a^{(t_1)}_i + \cdots 
                       + a^{(t_{N -1})}_i + (N - 1)(r-1) }
                       H(\Ba_i, k_i)\biggr)\Ba,   
\end{align*}
where $H(\Ba_i, k_i) = \prod_{1 \le j <  N}A(k_i - x_{ij}, r - a^{(t_j)}_i)$ 
with $x_{ij} = a^{(t_1)}_i + \cdots + a^{(t_{j-1})}_i$.  
(Here we write $A_i = \{ t_1, \dots, t_{N -1}\}$ along the natural order.)
\end{lem}

 \para{5.17}
In the special case where $\Bs = (L, \dots, L)$, we set 
$\Bk' = {}^t(L -k_1, \dots, L - k_m)$ for $\Bk = {}^t(k_1, \dots, k_m)$ with $k_i \le L$.
Then (5.16.1) is written as
\begin{align*}
\tag{5.17.1}
(\Bk, \bold 0, \dots, \bold 0, \Bk') = \sum_{\Ba \in \SE_{n,m}(L)}
            \biggl(\prod_{1 \le i \le m}(-1)^{a^{(t_1)}_i + \cdots 
                       + a^{(t_{N -1})}_i + (N - 1)(r-1)}
                       H(\Ba_i, k_i)\biggr)\Ba.   
\end{align*}

As in (5.11.3), $H(\Ba_i, k_i)$ can be written as

\begin{equation*}
\tag{5.17.2}
H(\Ba_i, k_i) = \prod_{1 \le j < N}A(k_i - x_{ij}, r- a^{(t_j)}_i)
                    = \sum_{j = 0}^{(N -1)(r-1)}G_j^{(i)}(q)q^{-k_i\bigl((r-1)(N -1) - 2j\bigr)}, 
\end{equation*}
where $G_j^{(i)}(q) \in \BQ(q)$ is independent from $k_i$. 
We shall prove the following formula.
\begin{prop}  
Set $L = (N -1)(r-1) + 1$.  The following formula holds in $E_{n,m}(L)$. 
\begin{equation*}
\tag{5.18.1}
\sum_{0 \le k_1 \le L}\cdots \sum_{0 \le k_m \le L}(-1)^{k_1 + \cdots + k_m}
              \begin{bmatrix}
                      L \\
                      k_1
             \end{bmatrix}
                \cdots 
             \begin{bmatrix}
                     L  \\
                     k_m
             \end{bmatrix}(\Bk, \bold 0, \cdots, \bold 0, \Bk') = 0,  
\end{equation*}
where $\Bk = {}^t(k_1, \dots, k_m)$ and $\Bk' = {}^t(L - k_1, \dots, L - k_m)$. 
\end{prop}

\begin{proof}
As in the proof of Proposition 5.11, by (5.17.1) and (5.17.2), the proof of 
(5.18.1) is reduced to the following formula; for $0  \le j_1, \dots, j_m \le L - 1$, 
\begin{equation*}
\tag{5.18.2}
\sum_{0 \le k_i \le L}(-1)^{k_1 + \cdots + k_m}
                      q^{-\bigl((L - 1 - 2j_1)k_1 + \cdots + (L  - 1-  2j_m )k_m\bigr)}
                  \begin{bmatrix}
                       L  \\
                       k_1
                  \end{bmatrix}
                    \cdots
                  \begin{bmatrix}
                       L \\
                       k_m
                  \end{bmatrix}  = 0.
\end{equation*}
But the left hand side of (5.18.2) is equal to 
\begin{equation*}
\prod_{i = 1}^m\biggl(\sum_{0 \le k_i \le L}(-1)^{k_i}q^{-(L - 1 - 2j_i)k_i}
                  \begin{bmatrix}
                    L \\
                    k_i
                  \end{bmatrix}\biggr), 
\end{equation*}
which is equal to zero by (5.6.5).  Thus (5.18.2) holds, and the proposition 
follows.   
\end{proof}

\para{5.19.}
Following the discussion in 5.14, 
we translate (5.18.1) to the original setup as in 5.1 and 5.14, namely, 
$\eta = \{ 1_1, \dots, 1_m \}$ $\eta' = \{2_1, \dots, 2_{n-1}\}$.
We follows the notation in 5.14. 
Recall that $d_i = (\a_i,\a_i)/2$ fo $i \in I$, and 
$d_{\eta} = (\a_{\eta}, \a_{\eta})_1/2$ for $\eta \in \ul I$.  
Note that if $i \in \eta$ as above, then $d_{\eta} = |\eta|d_i = md_i$.  
We consider $\wt f_{\eta} = \prod_{i \in \eta}f_i$, and 
$\wt f_{\eta'} = \prod_{j \in \eta'}f_j$.  
As a corollary to Proposition 5.18, we have

\begin{prop}  
Under the notation above, we have

\begin{equation*}
\tag{5.20.1}
\sum_{0 \le k_1 \le L}\cdots \sum_{0 \le k_m \le L}(-1)^{k_1 + \cdots + k_m}
               \begin{bmatrix}
                    L  \\
                    k_1
               \end{bmatrix}_{d_i}
                   \cdots 
               \begin{bmatrix}
                    L \\
                   k_m
               \end{bmatrix}_{d_i}
                  (f_{1_1}^{k_1}\cdots f_{1_m}^{k_m})\wt f_{\eta'}
                  (f_{1_1}^{k'_1}\cdots f_{1_m}^{k'_m}) = 0,   
\end{equation*}
where $k'_a = L  - k_a$ for $a = 1, \dots, m$, and $i \in \eta$. 
\end{prop}

\para{5.21.}
We consider the action of $\s$ on $I$, and on $\BU_q^-$.  
Then $\wt f_{\eta'}$ is $\s$-stable.  Since $\s$ maps 
$f_{1_i}^{k_i}$ to $f_{1_{i+1}}^{k_i}$, $f_{1_1}^{k_1}\cdots f_{1_m}^{k_m}$
is $\s$-stable if and only if $k_1 = \cdots = k_m$, namely 
$f_{1_1}^{k_1}\cdots f_{1_m}^{k_m} = \wt f_{\eta}^k$ for some $0 \le k \le L$.    
Thus (5.20.1) can be rewritten as
\begin{equation*}
\tag{5.21.1}
\sum_{0 \le k \le L}(-1)^{km}\biggl(\begin{bmatrix}
                                L   \\
                                k
                             \end{bmatrix}_{d_i}\biggr)^m
           \wt f_{\eta}^k\wt f_{\eta'} \wt f_{\eta}^{L -k} \equiv 0  \mod J_1,
\end{equation*}
where $J_1$ is a subspace of $\BU_q^-$ spanned by the orbit sum 
$O(x)$ for $x \in \BU_q^-$ such that $O(x) \ne x$.  
The following result proves (5.1.3), hence the proof of Proposition 3.5
is now complete.

\begin{prop}  
Under the setup in 5.1, (5.1.3) holds, 
namely, the equation
\begin{equation*}
\tag{5.22.1}
\sum_{k = 0}^{1 - a_{\eta\eta'}}(-1)^k\begin{bmatrix}
                                        1 - a_{\eta\eta'} \\
                                            k
                                     \end{bmatrix}_{d_{\eta}}
       \wt f_{\eta}^k\wt f_{\eta'} \wt f_{\eta}^{1 - a_{\eta\eta'} - k} 
          \equiv 0 \mod J
\end{equation*}  
holds in ${}_{\BA'}\BU_q^{-,\s}$. 
\end{prop}

\begin{proof}
We follow the notation in 5.21.
Since the order of $\s$ is a power of $p$, $m = |\eta|$ is also 
a power of $p$. Hence in the formula (5.21.1), we have 
an equality in $\BA'$,
\begin{equation*}
\biggl(\begin{bmatrix}
                L  \\
                k
        \end{bmatrix}_{d_i}\biggr)^m = 
             \begin{bmatrix}
                L  \\
                k
             \end{bmatrix}_{md_i}  = 
             \begin{bmatrix}
                L  \\
                k
             \end{bmatrix}_{d_{\eta}} 
\end{equation*}
since $d_{\eta} = md_i$.  Note that $(-1)^{km} = (-1)^k$ 
if $m$ is odd, and the term $(-1)^{k}$ is ignorable in ${}_{\BA'}\BU_q^-$ if 
$m$ is even.  Since $L = 1 - a_{\eta\eta'}$, the proposition follows from 
(5.21.1). 
\end{proof}

\remark{5.23.}
In the case where $X$ is of finite or affine type, and 
the order of $\s$ is 2 or 3, the isomorphism 
$\Phi: {}_{\BA'}\ul\BU_q^- \isom \BV_q$ was established in [SZ1, SZ2].  
The proof of the fact that $\Phi$ is a homomorphism 
is reduced to the case where $\ul X$ has rank 2, namely, 
$\ul X$ is of type $A_1 \times A_1, A_2, B_2$ or $G_2$. In [SZ1, Prop.1.10], 
this was verified by case by case computation.  The most complicated 
one is the case where $X$ is of type $D_4$ and 
$\ul X$ is of type $G_2$. (3.5.1) in [SZ1] corresponds to
the formula in Proposition 5.13, namely the case where $|\eta| = 1$, and $|\eta'| = 3$. 
(3.5.1) was proved after a hard computation by making use of PBW-bases of 
$\BU_q^-$ of type $D_4$.  As stated in Remark 3.6 in [SZ1], the formula (3.5.1) holds
without appealing modulo $J$ nor modulo 3, which corresponds to the statement
in Proposition 5.13. For the proof of (3.4.1) in [SZ1], which is the case 
where $|\eta| = 3$ and $|\eta'| = 1$, we need 
to consider in ${}_{\BA'}\BU_q^-$ with modulo $J$. 
This corresponds to the situation in Proposition 5.22. 
\par
Note that Proposition 5.13 and Proposition 5.22 can be proved for the general setup,
in a uniform way, and the discussion there is simpler, and more 
transparent  than the one used in [SZ1]. 

\par\bigskip


\par
\par\vspace{1.5cm}
\noindent
Y. Ma \\
School of Mathematical Sciences, Tongji University \\ 
1239 Siping Road, Shanghai 200092, P.R. China  \\
E-mail: \verb|1631861@tongji.edu.cn|

\par\vspace{0.5cm}
\noindent
T. Shoji \\
School of Mathematica Sciences, Tongji University \\ 
1239 Siping Road, Shanghai 200092, P.R. China  \\
E-mail: \verb|shoji@tongji.edu.cn|

\par\vspace{0.5cm}
\noindent
Z. Zhou \\
School of Mathematical Sciences, Tongji University \\ 
1239 Siping Road, Shanghai 200092, P.R. China  \\
E-mail: \verb|forza2p2h0u@163.com|

\end{document}